\documentclass[11pt]{amsart}
\usepackage[latin9]{inputenc}
\usepackage[a4paper]{geometry}
\usepackage{textcomp}
\usepackage{mathrsfs}
\usepackage{amstext}
\usepackage{amsthm}
\usepackage{amssymb}

\makeatletter
\numberwithin{equation}{section}
\numberwithin{figure}{section}

\usepackage[all]{xy}
\usepackage{indentfirst}
\usepackage{amscd}\usepackage{tikz}
\usepackage{tikz-cd}

\allowdisplaybreaks[2]

\newtheorem{theorem}{Theorem}[section]
\newtheorem{proposition}[theorem]{Proposition}\newtheorem{lemma}[theorem]{Lemma}\newtheorem{corollary}[theorem]{Corollary}\theoremstyle{definition}
\newtheorem{remark}[theorem]{Remark}\newtheorem{definition}[theorem]{Definition}\numberwithin{equation}{section}

\makeatother

\begin{document}
\title{Cohomology of a restricted Lie algebra with a restricted derivation
in characteristic 2}
\author{Dan Mao}
\address{D. Mao: School of Mathematics and Statistics, Northeast Normal University,
Changchun 130024, China}
\email{danmao678@nenu.edu.cn}
\author{Liangyun Chen$^{*}$}
\address{L. Chen: School of Mathematics and Statistics, Northeast Normal University,
Changchun 130024, China}
\email{chenly640@nenu.edu.cn}
\thanks{{*}Corresponding author.}
\thanks{\emph{MSC}(2020). 17B40, 17B50, 17B56, 16S20}
\thanks{\emph{Key words and phrases}. Restricted Lie Algebra, Restricted derivation,
Cohomology, Deformation, Central extension.}
\thanks{Supported by NSF of Jilin Province (No. YDZJ202201ZYTS589), NNSF of
China (Nos. 12271085, 12071405).}
\begin{abstract}
This paper mainly studies the ResLieDer pair in characteristic 2,
that is, a restricted Lie algebra with a restricted derivation. We
define the restricted representation of a ResLieDer pair and the corresponding
cohomology complex. We show that a ResLieDer pair is rigid if the
second cohomology group is trivial and a deformation of order $n$
is extensible if and only if its obstruction class is trivial. Moreover,
we prove that the central extensions of a ResLieDer pair are classified
by the second cohomology group. Finally, we show that a pair of restricted
derivations is extensible if and only if its obstruction class is
trivial.

\tableofcontents{}
\end{abstract}

\maketitle

\section{Introduction}

The definition of restricted Lie algebras was first introduced by
Jacobson \cite{Jac1}. It is a modular Lie algebra $\mathfrak{g}$
equipped with a $p$-mapping $[p]:\mathfrak{g}\rightarrow\mathfrak{g}$
which has the same basic properties as the Frobenius mapping $x\rightarrow x^{p}$
of an associative algebra. Since restricted Lie algebra plays an important
role in the classification of simple modular Lie algebras \cite{BW,BW1,Str,Str2,Str3},
it has been extensively studied \cite{EFi,KJ,MS,SU}, especially the
restricted cohomology theory of it \cite{EM,EF,Fel,Hoc,May}. In \cite{Hoc},
the notion of restricted universal enveloping algebra was used to
define the restricted cohomology groups of restricted Lie algebras.
In \cite{Fel}, some vanishing and non-vanishing theorems for the
complete restricted cohomology of restricted Lie algebras were proved.
In \cite{EF}, the authors provided a suitably small complex for the
computation of restricted cohomology. However, since the non-linearity
of the $p$-mapping leads to the difficulty of the study, they only
provided the description of the first and second restricted cohomology
groups of a general restricted Lie algebra. In a simpler case of an
abelian restricted Lie algebra, the description of restricted cohomology
groups is up to dimension $p$. Recently, a complete cohomology complex
of restricted Lie algebras in charactersitic 2 has been introduced
by \cite{EM}, which has no analog in characteristic $p>2$.

Given algebraic objects, they can also become useful via derivations.
For example, \cite{Vor} provided a construction of homotopy Lie algebras
with the use of higher derived brackets. \cite{BV} showed that algebras
and their derivations played an essential role in the study of gauge
theories in quantum field theory. Moreover, algebras and their derivations
are also useful in control theory \cite{AKd}, deformation formulas
\cite{CGG} and differential Galois theory \cite{Mag}. In \cite{Lod},
the operad of associative algebras with derivations was studied. In
\cite{TFS}, the authors constructed the cohomologies of a Lie algebra
with a derivation (i.e. LieDer pair) and applied them to study the
deformation and extension of LieDer pairs. Afterwards, the results
in \cite{TFS} were extended to other algebras, such as associative
algebras \cite{DM}, Leibniz algebras \cite{Das}, 3-Lie algebras
\cite{Xu}, Lie triple systems \cite{Guo} and Leibniz triple systems
\cite{WMSC}, with derivations.

Inspired by the above facts and results, this paper studies the restricted
Lie algebra with a restricted derivation in characteristic 2, which
is called a ResLieDer pair.

The reason why we study the ResLieDer pair in characteristic 2 is
that only the general (nonabelian) restricted Lie algebra in characteristic
2 has a complete cohomology complex, which will be used to construct
a complete cohomology complex for a ResLieDer pair.

This paper is organized as follows. In Section 2, we first recall
some basic definitions on restricted Lie algebras. Then, the definitions
of a ResLieDer pair and its restricted representation are introduced.
We prove that given a restricted representation $(V,\rho,\eta_{V})$
of a ResLieDer pair $((\mathfrak{g},[2]),\mathscr{D})$, one can obtain
the semidirect product ResLieDer pair $((\mathfrak{g}\oplus V,[2]_{\rho}),\mathscr{D}+\eta_{V})$
(Proposition 2.14). In Section 3, we first recall the cohomology of
a restricted Lie algebra $(\mathfrak{g},[2])$ with coefficients in
the restricted $\mathfrak{g}$-module $(V,\rho)$ in characteristic
2. Then, a cohomology complex of a ResLieDer pair $((\mathfrak{g},[2]),\mathscr{D})$
associated to its restricted representation $(V,\rho,\eta_{V})$ is
defined. In Section 4, the deformations of a ResLieDer pair are studied.
We show that the infinitesimal deformations modulo equivalences are
controlled by the second cohomology groups with values in the adjoint
representation (Theorem 4.6) and the ResLieDer pair $((\mathfrak{g},[2]),\mathscr{D})$
is rigid if the second cohomology group is trivial ( Theorem 4.9).
We also show that a deformation of order $n$ is extensible if and
only if its obstruction class is trivial (Theorem 4.13). In Section
5, we study the central extensions of a ResLieDer pair. We show that
the central extensions of a ResLieDer pair $((\mathfrak{g},[2]_{\mathfrak{g}}),\mathscr{D}_{\mathfrak{g}})$
by a strongly abelian ResLieDer pair $((\mathfrak{h},[2]_{\mathfrak{h}}),\mathscr{D}_{\mathfrak{h}})$
are classified by the second cohomology group with values in the trivial
representation (Theorem 5.7). In Section 6, the extension problem
of a pair of restricted derivations are considered.

Throughout the paper, $\mathbb{F}$ denotes a field of characteristic
2. And all the results obtained in this paper are over the field in
characteristic 2.

\section{ResLieDer pairs and their representations}

In this section, we first recall some basic definitions on restricted
Lie algebras. Then the definitions of a ResLieDer pair and its restricted
representation are introduced.

\begin{definition} (see \cite{SF}) Let $\mathfrak{g}$ be a Lie
algebra over $\mathbb{K}$, which is a field of characteristic $p>0$.
A mapping $[p]:\mathfrak{g}\rightarrow\mathfrak{g}$, $x\mapsto x^{[p]}$
is called a $p$-mapping, if

(1) $\mathrm{ad}x^{[p]}=(\mathrm{ad}x)^{p},\;\forall x\in\mathfrak{g}$,

(2) $(kx)^{[p]}=k^{p}x^{[p]},\forall x\in\mathfrak{g},k\in\mathbb{K}$,

(3) $(x+y)^{[p]}=x^{[p]}+y^{[p]}+\sum_{i=1}^{p-1}s_{i}(x,y)$, where
$(\mathrm{ad}(x\otimes X+y\otimes1))^{p-1}(x\otimes1)=\sum_{i=1}^{p-1}is_{i}(x,y)\otimes X^{i-1}$
in $\mathfrak{g}\otimes_{\mathbb{K}}\mathbb{K}[X]$, $\forall x,y\in\mathfrak{g}$.
The pair $(\mathfrak{g},[p])$ is referred to as a restricted Lie
algebra.

\end{definition}

\begin{definition} (see \cite{Hoc}) A restricted Lie algebra $(\mathfrak{g},[p])$
is said to be strongly abelian if $[\mathfrak{g},\mathfrak{g}]_{\mathfrak{g}}=0$
and $\mathfrak{g}^{[p]}=0$.

\end{definition}

\begin{definition} (see \cite{SF}) Let $(\mathfrak{g},[p])$ be
a restricted Lie algebra. A subalgebra $\mathfrak{h}\subset\mathfrak{g}$
is called a $p$-subalgebra if $x^{[p]}\in\mathfrak{h}$, $\forall x\in\mathfrak{h}$.

\end{definition}

\begin{definition} (see \cite{SF}) Let $(\mathfrak{g},[p])$ be
a restricted Lie algebra in prime characteristic $p$. A derivation
$\mathscr{D}\in\mathrm{Der}(\mathfrak{g})$ is called restricted if

\[
\mathscr{D}(x^{[p]})=(\mathrm{ad}x)^{p-1}(\mathscr{D}(x)),\;\forall x\in\mathfrak{g}.
\]
\end{definition}

Denoted by $\mathrm{Der}^{p}(\mathfrak{g})$ the set of all restricted
derivations of $\mathfrak{g}$.

Note that if $p=2$, then a restricted derivation $\mathscr{D}\in\mathrm{Der}^{2}(\mathfrak{g})$
satisfies $\mathscr{D}(x^{[2]})=[x,\mathscr{D}(x)]_{\mathfrak{g}}$.

\begin{definition} (see \cite{SF}) Let $(\mathfrak{g},[p])$ be
a restricted Lie algebra and $V$ a vector space in prime characteristic
$p$. A restricted representation of $\mathfrak{g}$ in $V$ is a
homomorphism $\rho:\mathfrak{g}\rightarrow\mathfrak{gl}(V)$ with
$\rho(x^{[p]})=\rho(x)^{p}$, $\forall x\in\mathfrak{g}$. It is denoted
by $(V,\rho)$.

If $(V,\rho)$ is a restricted representation of $(\mathfrak{g},[p])$,
then $V$ is called a restricted $\mathfrak{g}$-module.

\end{definition}

\begin{remark} If $\rho=0$, then $(V,0)$ is called a trivial representation
of $(\mathfrak{g},[p])$ and $V$ is called a trivial $\mathfrak{g}$-module.

\end{remark}

\begin{definition} A ResLieDer pair in prime characteristic $p$
is a restricted Lie algebra $(\mathfrak{g},[p])$ with a restricted
derivation $\mathscr{D}\in\mathrm{Der}^{p}(\mathfrak{g})$. It is
denoted by $((\mathfrak{g},[p]),\mathscr{D})$.

\end{definition}

\begin{definition} Let $((\mathfrak{g},[p]),\mathscr{D})$ be a ResLieDer
pair.

(1) It is said to be abelian if the Lie bracket $[\cdot,\cdot]_{\mathfrak{g}}$
in $\mathfrak{g}$ is trivial, that is, $[\mathfrak{g},\mathfrak{g}]_{\mathfrak{g}}=0$;

(2) It is said to be strongly abelian if $[\mathfrak{g},\mathfrak{g}]_{\mathfrak{g}}=0$
and $\mathfrak{g}^{[p]}=0$.

\end{definition}

\begin{definition} Let $((\mathfrak{g},[p]_{\mathfrak{g}}),\mathscr{D}_{\mathfrak{g}})$
and $((\mathfrak{h},[p]_{\mathfrak{h}}),\mathscr{D}_{\mathfrak{h}})$
be ResLieDer pairs. A morphism $\pi$ from $((\mathfrak{g},[p]_{\mathfrak{g}}),\mathscr{D}_{\mathfrak{g}})$
to $((\mathfrak{h},[p]_{\mathfrak{h}}),\mathscr{D}_{\mathfrak{h}})$
is a Lie algebra morphism $\pi:\mathfrak{g}\rightarrow\mathfrak{h}$
such that
\[
\pi(x^{[p]_{\mathfrak{g}}})=\pi(x)^{[p]_{\mathfrak{h}}}\quad\mathrm{and}\quad\pi\text{\textopenbullet}\mathscr{D}_{\mathfrak{g}}(x)=\mathscr{D}_{\mathfrak{h}}\text{\textopenbullet}\pi(x),\quad\forall x\in\mathfrak{g}.
\]

\end{definition}

\begin{definition} A restricted representation of a ResLieDer pair
$((\mathfrak{g},[p]),\mathscr{D})$ on a vector space $V$ with respect
to $\eta_{_{V}}\in\mathrm{gl}(V)$ is a Lie algebra morphism $\rho:\mathfrak{g}\rightarrow\mathrm{gl}(V)$
such that
\[
\rho(x^{[p]})=\rho(x)^{p},\;\eta_{V}\text{\textopenbullet}\rho(x)=\rho(\mathscr{D}(x))+\rho(x)\text{\textopenbullet}\eta_{V},\;\forall x\in\mathfrak{g}.
\]
Denote by $(V,\rho,\eta_{V})$ a restricted representation.

\end{definition}

\begin{remark} (1) For all $x\in\mathfrak{g}$, we define $\mathrm{ad}x:\mathfrak{g}\rightarrow\mathfrak{g}$
by $\mathrm{ad}x(y)=[x,y]_{\mathfrak{g}}$, $\forall y\in\mathfrak{g}$.
It can be proved that $\mathrm{ad}$ is a restricted representation
of the ResLieDer pair $((\mathfrak{g},[p]),\mathscr{D})$ on $\mathfrak{g}$
with respect to $\mathscr{D}$. The triple $(\mathfrak{g},\mathrm{ad},\mathscr{D})$
is called an adjoint representation.

(2) If $\rho=0$, then the triple $(V,0,\eta_{V})$ is called a trivial
representation.

\end{remark}

\begin{definition} Let $(V,\rho,\eta_{V})$ and $(V',\rho',\eta_{V'})$
be two restricted representations of a ResLieDer pair $((\mathfrak{g},[p]),\mathscr{D})$.
A morphism from $(V,\rho,\eta_{V})$ to $(V',\rho',\eta_{V'})$ is
a morphism $f:V\rightarrow V'$ such that
\[
f\text{\textopenbullet}\eta_{V}=\eta_{V'}\text{\textopenbullet}f.
\]

\end{definition}

\begin{lemma} (see \cite{Hum}) Let $\mathfrak{g}$ be a Lie algebra
and $(\rho,V)$ a representation of $\mathfrak{g}$. Then the space
$\mathfrak{g}\oplus V$ becomes a Lie algebra with the bracket
\[
[x+u,y+v]_{\rho}=[x,y]_{\mathfrak{g}}+\rho(x)(v)-\rho(y)(u),\;\forall x,y\in\mathfrak{g},u,v\in V.
\]
Denote by $\mathfrak{g}\ltimes V$ this Lie algebra.

\end{lemma}

\begin{proposition} Given a restricted representation $(V,\rho,\eta_{V})$
of a ResLieDer pair $((\mathfrak{g},[2]),\mathscr{D})$ over $\mathbb{F}$,
define $[2]_{\rho}:\mathfrak{g}\oplus V\rightarrow\mathfrak{g}\oplus V$
by
\[
(x+u)^{[2]_{\rho}}=x^{[2]}+\rho(x)(u)
\]
and $\mathscr{D}+\eta_{V}:\mathfrak{g}\oplus V\rightarrow\mathfrak{g}\oplus V$
as follows
\[
(\mathscr{D}+\eta_{V})(x+u)=\mathscr{D}(x)+\eta_{V}(u),\;\forall x\in\mathfrak{g},u\in V.
\]

Then $((\mathfrak{g}\oplus V,[2]_{\rho}),\mathscr{D}+\eta_{V})$ is
a ResLieDer pair with the Lie structure in the above lemma. It is
called the semidirect product of the ResLieDer pair $((\mathfrak{g},[2]),\mathscr{D})$
by the restricted representation $(V,\rho,\eta_{V})$ and denoted
by $\mathfrak{g}\ltimes_{\mathrm{ResLD}}V$.

\end{proposition}
\begin{proof}
First, it can be proved that $[2]_{\rho}$ is a 2-mapping on $\mathfrak{g}\oplus V$.
Indeed, for any $x,y\in\mathfrak{g}$ and $u,v\in V$, it follows
from
\[
[(x+u)^{[2]_{\rho}},y+v]_{\rho}=[x^{[2]}+\rho(x)(u),y+v]_{\rho}=[x^{[2]},y]_{\mathfrak{g}}+\rho(x^{[2]})(v)+\rho(y)(\rho(x)(u))
\]
and
\begin{align*}
 & [x+u,[x+u,y+v]_{\rho}]_{\rho}\\
= & [x+u,[x,y]_{\mathfrak{g}}+\rho(x)(v)+\rho(y)(u)]_{\rho}\\
= & [x,[x,y]_{\mathfrak{g}}]_{\mathfrak{g}}+\rho(x)^{2}(v)+\rho(x)(\rho(y)(u))+\rho([x,y]_{\mathfrak{g}})(u)\\
= & [x^{[2]},y]_{\mathfrak{g}}+\rho(x^{[2]})(v)+\rho(x)(\rho(y)(u))+\rho(x)(\rho(y)(u))+\rho(y)(\rho(x)(u))\\
= & [x^{[2]},y]_{\mathfrak{g}}+\rho(x^{[2]})(v)+\rho(y)(\rho(x)(u))
\end{align*}
that $[(x+u)^{[2]_{\rho}},y+v]_{\rho}=[x+u,[x+u,y+v]_{\rho}]_{\rho}$.

Moreover, we have
\begin{align*}
((x+u)+(y+v))^{[2]_{\rho}} & =((x+y)+(u+v))^{[2]_{\rho}}=(x+y)^{[2]}+\rho(x+y)(u+v)\\
 & =x^{[2]}+y^{[2]}+[x,y]_{\mathfrak{g}}+\rho(x)(u)+\rho(x)(v)+\rho(y)(u)+\rho(y)(v)\\
 & =(x^{[2]}+\rho(x)(u))+(y^{[2]}+\rho(y)(v))+([x,y]_{\mathfrak{g}}+\rho(x)(v)+\rho(y)(u))\\
 & =(x+u)^{[2]_{\rho}}+(y+v)^{[2]_{\rho}}+[x+u,y+v]_{\rho}.
\end{align*}

For any $x\in\mathfrak{g}$, $u\in V$ and $a\in\mathbb{F}$, we
\[
(a(x+u))^{[2]_{\rho}}=(ax+au)^{[2]_{\rho}}=(ax)^{[2]}+\rho(ax)(au)=a^{2}(x^{[2]}+\rho(x)(u))=a^{2}(x+u)^{[2]_{\rho}}.
\]
Therefore, $[2]_{\rho}$ is a 2-mapping on $\mathfrak{g}\oplus V$.

It remains to prove that $\mathscr{D}+\eta_{V}$ is a restricted derivation.
Indeed, for any $x,y\in\mathfrak{g},u,v\in V$, we have
\begin{align*}
 & (\mathscr{D}+\eta_{V})([x+u,y+v]_{\rho})\\
= & (\mathscr{D}+\eta_{V})([x,y]_{\mathfrak{g}}+\rho(x)(v)+\rho(y)(u))=\mathscr{D}([x,y]_{\mathfrak{g}})+\eta_{V}(\rho(x)(v)+\rho(y)(u))\\
= & [\mathscr{D}(x),y]_{\mathfrak{g}}+[x,\mathscr{D}(y)]_{\mathfrak{g}}+(\rho(\mathscr{D}(x))(v)+\rho(x)\text{\textopenbullet}\eta_{V}(v))+(\rho(\mathscr{D}(y))(u)+\rho(y)\text{\textopenbullet}\eta_{V}(u))\\
= & ([\mathscr{D}(x),y]_{\mathfrak{g}}+\rho(\mathscr{D}(x))(v)+\rho(y)\text{\textopenbullet}\eta_{V}(u))+([x,\mathscr{D}(y)]_{\mathfrak{g}}+\rho(\mathscr{D}(y))(u)+\rho(x)\text{\textopenbullet}\eta_{V}(v))\\
= & [(\mathscr{D}+\eta_{V})(x+u),y+v]_{\rho}+[x+u,(\mathscr{D}+\eta_{V})(y+v)]_{\rho},
\end{align*}
which implies that $\mathscr{D}+\eta_{V}$ is a derivation. Since
\[
(\mathscr{D}+\eta_{V})((x+u)^{[2]_{\rho}})=(\mathscr{D}+\eta_{V})(x{}^{[2]}+\rho(x)(u))=\mathscr{D}(x^{[2]})+\eta_{V}(\rho(x)(u))
\]
and
\begin{align*}
 & [x+u,(\mathscr{D}+\eta_{V})(x+u)]_{\rho}\\
= & [x+u,\mathscr{D}(x)+\eta_{V}(u)]_{\rho}=[x,\mathscr{D}(x)]_{\mathfrak{g}}+\rho(x)(\eta_{V}(u))+\rho(\mathscr{D}(x))(u)\\
= & \mathscr{D}(x^{[2]})+\eta_{V}(\rho(x)(u)),
\end{align*}
we have $(\mathscr{D}+\eta_{V})((x+u)^{[2]_{\rho}})=[x+u,(\mathscr{D}+\eta_{V})(x+u)]_{\rho}$,
which implies that the derivation $\mathscr{D}+\eta_{V}$ is restricted.
The proof is complete.
\end{proof}

\section{Cohomology of a ResLieDer pair}

Let's first recall the cohomology of a restricted Lie algebra $(\mathfrak{g},[2])$
with coefficients in the restricted $\mathfrak{g}$-module $(V,\rho)$
in characteristic 2 \cite{EM}.

For $n\geq0$, denote by $C_{*2}^{n}(\mathfrak{g};V)$ the space of
$n$-cochains on $\mathfrak{g}$ with values in $V$. Then $C_{*2}^{0}(\mathfrak{g};V)=V$,
$C_{*2}^{1}(\mathfrak{g};V)=\mathrm{Hom}(\mathfrak{g};V)$ and
\begin{align*}
C_{*2}^{2}(\mathfrak{g};V)= & \{(\varphi_{2},\omega_{2})|\varphi_{2}:\land^{2}\mathfrak{g}\rightarrow M\mathrm{\;is\;a\;bilinear\;map,\omega_{2}:\mathfrak{g}\rightarrow M\;\mathrm{satisfying}}\\
 & \;\omega_{2}(ax)=a^{2}\omega_{2}(x)\;\mathrm{and}\;\omega_{2}(x+y)=\omega_{2}(x)+\omega_{2}(y)+\varphi_{2}(x,y),\\
 & \;\mathrm{for\;all}\;x,y\in\mathfrak{g}\;\mathrm{and}\;\mathrm{all}\;a\in\mathbb{F}\}.
\end{align*}

Moreover, for $n\geq3$, a $n$-cochain of the cohomology is a pair
$(\varphi_{n},\omega_{n})$ consisting of a $n$-linear map $\varphi_{n}:\land^{n}\mathfrak{g}\rightarrow V$
and a map $\omega_{n}:\land^{n-1}\mathfrak{g}\rightarrow V$, which
satisfies the following conditions (for all $x,y,z_{i},z_{i'}\in\mathfrak{g},2\leq i\leq n-1$
and all $a,b\in\mathbb{F}$):

\[
\omega_{n}(ax,z_{2},\cdot\cdot\cdot,z_{n-1})=a^{2}\omega_{n}(x,z_{2},\cdot\cdot\cdot,z_{n-1});
\]
\[
\omega_{n}(x,z_{2},\cdot\cdot\cdot,az_{i}+bz_{i'},\cdot\cdot\cdot,z_{n-1})=a\omega_{n}(x,z_{2},\cdot\cdot\cdot,z_{i},\cdot\cdot\cdot,z_{n-1})+b\omega_{n}(x,z_{2},\cdot\cdot\cdot,z_{i'},\cdot\cdot\cdot,z_{n-1});
\]
\[
\omega_{n}(x+y,z_{2},\cdot\cdot\cdot,z_{n-1})=\omega_{n}(x,z_{2},\cdot\cdot\cdot,z_{n-1})+\omega_{n}(y,z_{2},\cdot\cdot\cdot,z_{n-1})+\varphi_{n}(x,y,z_{2},\cdot\cdot\cdot,z_{n-1}).
\]

The coboundary operator $\partial^{0}:C_{*2}^{0}(\mathfrak{g};V)\rightarrow C_{*2}^{1}(\mathfrak{g};V)$
is defined by $\partial^{0}:v\mapsto d^{0}v$, where $d^{0}v(x)=\rho(x)(v),$
for any $x\in\mathfrak{g}$. The map $\partial^{1}:C_{*2}^{1}(\mathfrak{g};V)\rightarrow C_{*2}^{2}(\mathfrak{g};V)$
is given by $\partial^{1}:\varphi_{1}\mapsto(d^{1}\varphi_{1},\omega_{\varphi_{1}}),$where
\begin{align*}
d^{1}\varphi_{1}(x,y) & =\varphi_{1}([x,y]_{\mathfrak{g}})+\rho(x)(\varphi_{1}(y))+\rho(y)(\varphi_{1}(x)),\;\forall x,y\in\mathfrak{g},\\
\omega_{\varphi_{1}}(x) & =\varphi_{1}(x^{[2]})+\rho(x)(\varphi_{1}(x)),\;\forall x\in\mathfrak{g}.
\end{align*}

\begin{remark} The notation $\omega_{\varphi_{1}}$ represents a
map defined by the linear map $\varphi_{1}:\mathfrak{g}\rightarrow V$
as above such that $(d^{1}\varphi_{1},\omega_{\varphi_{1}})$ becomes
a 2-cochain. If $\omega$ is replaced by another letter, for example
$\varsigma$, then $\varsigma_{\varphi_{1}}(x)=\omega_{\varphi_{1}}(x)$,
$\forall x\in\mathfrak{g}$.

\end{remark}

For $n\geq2$, the coboundary operator $\partial^{n}:C_{*2}^{n}(\mathfrak{g};V)\rightarrow C_{*2}^{n+1}(\mathfrak{g};V)$
is given by $\partial^{n}:(\varphi_{n},\omega_{n})\mapsto(d^{n}\varphi_{n},d^{n}\omega_{n}),$where

\begin{align*}
d^{n}\varphi_{n}(z_{1},\cdot\cdot\cdot,z_{n+1})= & \sum_{1\leq i\leq n+1}\rho(z_{i})(\varphi_{n}(z_{1},\cdot\cdot\cdot,\hat{z_{i}},\cdot\cdot\cdot,z_{n+1}))\\
 & +\sum_{1\leq i<j\leq n+1}\varphi_{n}([z_{i},z_{j}]_{\mathfrak{g}},z_{1},\cdot\cdot\cdot,\hat{z_{i}},\cdot\cdot\cdot,\hat{z_{j}},\cdot\cdot\cdot,z_{n+1}),
\end{align*}

\begin{align*}
d^{n}\omega_{n}(x,z_{1},\cdot\cdot\cdot,z_{n-1})= & \rho(x)(\varphi_{n}(x,z_{1},\cdot\cdot\cdot,z_{n-1}))+\varphi_{n}(x^{[2]_{\alpha}},z_{1},\cdot\cdot\cdot,z_{n-1})\\
 & +\sum_{1\leq i\leq n-1}\varphi_{n}([x,z_{i}]_{\mathfrak{g}},x,z_{1},\cdot\cdot\cdot,\hat{z_{i}},\cdot\cdot\cdot,z_{n-1})\\
 & +\sum_{1\leq i\leq n-1}\rho(z_{i})(\omega_{n}(x,z_{1},\cdot\cdot\cdot,\hat{z_{i}},\cdot\cdot\cdot,z_{n-1}))\\
 & +\sum_{1\leq i<j\leq n-1}\omega_{n}(x,[z_{i},z_{j}]_{\mathfrak{g}},z_{1},\cdot\cdot\cdot,\hat{z_{i}},\cdot\cdot\cdot,\hat{z_{j}},\cdot\cdot\cdot,z_{n-1}).
\end{align*}

Denote by $Z_{*2}^{n}(\mathfrak{g};V)$ the set of $n$-cocycles and
$B_{*2}^{n}(\mathfrak{g};V)$ the set of $n$-coboundaries. And the
corresponding cohomology group is denoted by $H_{*2}^{n}(\mathfrak{g};V)=Z_{*2}^{n}(\mathfrak{g};V)/B_{*2}^{n}(\mathfrak{g};V)$.

In the following, the cohomology of a ResLieDer pair $((\mathfrak{g},[2]),\mathscr{D})$
with values in its restricted representation $(V,\rho,\eta_{V})$
will be given.

For $n\geq0$, denote by $C_{\mathrm{ResLD}}^{n}(\mathfrak{g};V)$
the space of ResLieDer pair $n$-cochains on $((\mathfrak{g},[2]),\mathscr{D})$
with values in $V$. Let $C_{\mathrm{ResLD}}^{0}(\mathfrak{g};V)=0$
and $C_{\mathrm{ResLD}}^{1}(\mathfrak{g};V)=\mathrm{Hom}(\mathfrak{g};V)$.
For $n\geq2$, the space of ResLieDer pair $n$-cochains is defined
by
\[
C_{\mathrm{ResLD}}^{n}(\mathfrak{g};V):=C_{*2}^{n}(\mathfrak{g};V)\times C_{*2}^{n-1}(\mathfrak{g};V).
\]

For $n\geq1$, we define an operator $\delta:C_{*2}^{n}(\mathfrak{g};V)\rightarrow C_{*2}^{n}(\mathfrak{g};V)$.
Details are as follows: The operator $\delta:C_{*2}^{1}(\mathfrak{g};V)\rightarrow C_{*2}^{1}(\mathfrak{g};V)$
is given by
\[
(\delta\varphi_{1})(x)=\eta_{V}\text{\textopenbullet}\varphi_{1}(x)+\varphi_{1}(\mathscr{D}(x)).
\]
The operator $\delta:C_{*2}^{2}(\mathfrak{g};V)\rightarrow C_{*2}^{2}(\mathfrak{g};V)$
is given by
\begin{align*}
(\delta\varphi_{2})(x,y) & =\eta_{V}\text{\textopenbullet}\varphi_{2}(x,y)+\varphi_{2}(\mathscr{D}(x),y)+\varphi_{2}(x,\mathscr{D}(y)),\\
(\delta\omega_{2})(x) & =\eta_{V}\text{\textopenbullet}\omega_{2}(x)+\varphi_{2}(x,\mathscr{D}(x)).
\end{align*}
And the operator $\delta:C_{*2}^{n}(\mathfrak{g};V)\rightarrow C_{*2}^{n}(\mathfrak{g};V)$,
$n\geq3$ is given by
\[
(\delta\varphi_{n})(z_{1},z_{2},\cdots,z_{n})=\eta_{V}\text{\textopenbullet}\varphi_{n}(z_{1},z_{2},\cdots,z_{n})+\sum_{i=1}^{n}\varphi_{n}(z_{1},\cdots,\mathscr{D}(z_{i}),\cdots,z_{n}),
\]

\begin{align*}
(\delta\omega_{n})(x,z_{2},\cdots,z_{n-1})= & \eta_{V}\text{\textopenbullet}\omega_{n}(x,z_{2},\cdots,z_{n-1})+\sum_{i=2}^{n-1}\omega_{n}(x,z_{2},\cdots,\mathscr{D}(z_{i}),\cdots,z_{n-1})\\
 & +\varphi_{n}(x,\mathscr{D}(x),z_{2},\cdots,z_{n-1}).
\end{align*}
Furthermore, define $\vartheta^{1}:C_{\mathrm{ResLD}}^{1}(\mathfrak{g};V)\rightarrow C_{\mathrm{ResLD}}^{2}(\mathfrak{g};V)$
by

\[
\vartheta^{1}\varphi_{1}=((d^{1}\varphi_{1},\omega_{\varphi_{1}}),\delta\varphi_{1}),\;\forall\varphi_{1}\in\mathrm{Hom}(\mathfrak{g};V)
\]
and $\vartheta^{2}:C_{\mathrm{ResLD}}^{2}(\mathfrak{g};V)\rightarrow C_{\mathrm{ResLD}}^{3}(\mathfrak{g};V)$
by

\[
\vartheta^{2}((\varphi_{2},\omega_{2}),\varphi_{1})=((d^{2}\varphi_{2},d^{2}\omega_{2}),(d^{1}\varphi_{1}+\delta\varphi_{2},\omega_{\varphi_{1}}+\delta\omega_{2})),
\]
for any $(\varphi_{2},\omega_{2})\in C_{*2}^{2}(\mathfrak{g};V),\varphi_{1}\in\mathrm{Hom}(\mathfrak{g};V)$.
For $n\geq3$, the map $\vartheta^{n}:C_{\mathrm{ResLD}}^{n}(\mathfrak{g};V)\rightarrow C_{\mathrm{ResLD}}^{n+1}(\mathfrak{g};V)$
is defined by

\[
\vartheta^{n}((\varphi_{n},\omega_{n}),(\varphi_{n-1},\omega_{n-1}))=((d^{n}\varphi_{n},d^{n}\omega_{n}),(d^{n-1}\varphi_{n-1}+\delta\varphi_{n},d^{n-1}\omega_{n-1}+\delta\omega_{n})),
\]
for any $(\varphi_{n},\omega_{n})\in C_{*2}^{n}(\mathfrak{g};V)$,
$(\varphi_{n-1},\omega_{n-1})\in C_{*2}^{n-1}(\mathfrak{g};V)$.

\begin{lemma} For $n\geq1$, the map $\partial^{n}$ and $\delta$
are commutative with each other, that is, $\partial^{n}\text{\textopenbullet}\delta=\delta\text{\textopenbullet}\partial^{n}$.
More precisely, we have $\delta(d^{n}\varphi_{n})=d^{n}(\delta\varphi_{n})$
and $\delta\omega_{\varphi_{1}}=\omega_{\delta\varphi_{1}}$, $\delta(d^{k}\omega_{k})=d^{k}(\delta\omega_{k})$,
for $k\geq3$.

\end{lemma}
\begin{proof}
It suffices to prove that $\delta\omega_{\varphi_{1}}=\omega_{\delta\varphi_{1}}$
and $\delta(d^{k}\omega_{k})=d^{k}(\delta\omega_{k})$, for $k\geq3$.

Since $\eta_{V}\text{\textopenbullet}\rho(x)=\rho(\mathscr{D}(x))+\rho(x)\circ\eta_{V}$
and $\mathscr{D}(x^{[2]})=[x,\mathscr{D}(x)]_{\mathfrak{g}}$, $\forall x\in\mathfrak{g}$,
we have
\begin{align*}
\omega_{\delta\varphi_{1}}(x) & =(\delta\varphi_{1})(x^{[2]})+\rho(x)((\delta\varphi_{1})(x))\\
 & =\eta_{V}\circ\varphi_{1}(x^{[2]})+\varphi_{1}(\mathscr{D}(x^{[2]}))+\rho(x)(\eta_{V}(\varphi_{1}(x)))+\rho(x)(\varphi_{1}(\mathscr{D}(x)))
\end{align*}
and
\begin{align*}
 & \delta\omega_{\varphi_{1}}(x)\\
= & \eta_{V}\circ\omega_{\varphi_{1}}(x)+d^{1}\varphi_{1}(x,\mathscr{D}(x))\\
= & \eta_{V}\circ\varphi_{1}(x^{[2]})+\eta_{V}(\rho(x)(\varphi_{1}(x)))+\rho(x)(\varphi_{1}(\mathscr{D}(x)))+\rho(\mathscr{D}(x))(\varphi_{1}(x))+\varphi_{1}([x,\mathscr{D}(x)]_{\mathfrak{g}})\\
= & \eta_{V}\circ\varphi_{1}(x^{[2]})+\varphi_{1}(\mathscr{D}(x^{[2]}))+(\eta_{V}(\rho(x)(\varphi_{1}(x)))+\rho(\mathscr{D}(x))(\varphi_{1}(x)))+\rho(x)(\varphi_{1}(\mathscr{D}(x)))\\
= & \eta_{V}\circ\varphi_{1}(x^{[2]})+\varphi_{1}(\mathscr{D}(x^{[2]}))+\rho(x)(\eta_{V}(\varphi_{1}(x)))+\rho(x)(\varphi_{1}(\mathscr{D}(x))).
\end{align*}
It follows that $\delta\omega_{\varphi_{1}}=\omega_{\delta\varphi_{1}}$.

Moreover, for any $x,z_{2},\cdots z_{k}\in\mathfrak{g}$, we have
\begin{align}
 & d^{k}(\delta\omega_{k})(x,z_{2},\cdots,z_{k})\nonumber \\
= & \rho(x)(\delta\varphi_{k}(x,z_{2},\cdots,z_{k}))+\delta\varphi_{k}(x^{[2]},z_{2},\cdots,z_{k})\\
 & +\sum_{i=2}^{k}\delta\varphi_{k}([x,z_{i}]_{\mathfrak{g}},x,z_{2},\cdots,\hat{z}_{i},\cdots,z_{k})+\sum_{i=2}^{k}\rho(z_{i})(\delta\omega_{k}(x,z_{2},\cdots,\hat{z}_{i},\cdots,z_{k}))\nonumber \\
 & +\sum_{2\leq i<j\leq k}\delta\omega_{k}(x,[z_{i},z_{j}]_{\mathfrak{g}},z_{2},\cdots,\hat{z}_{i},\cdots,\hat{z}_{j},\cdots,z_{k})\nonumber
\end{align}
and
\begin{align}
 & \delta(d^{k}\omega_{k})(x,z_{2},\cdots,z_{k})\nonumber \\
= & \eta_{V}\circ(d^{k}\omega_{k})(x,z_{2},\cdots,z_{k})+\sum_{i=2}^{k}(d^{k}\omega_{k})(x,z_{2},\cdots,\mathscr{D}(z_{i}),\cdots,z_{k})\\
 & +d^{k}\varphi_{k}(x,\mathscr{D}(x),z_{2},\cdots,z_{k}).\nonumber
\end{align}

By expanding (3.1) and (3.2) further, it follows from $\eta_{V}\text{\textopenbullet}\rho(x)=\rho(\mathscr{D}(x))+\rho(x)\circ\eta_{V}$,
$\mathscr{D}([z_{i},z_{j}]_{\mathfrak{g}})=[\mathscr{D}(z_{i}),z_{j}]_{\mathfrak{g}}+[z_{i},\mathscr{D}(z_{j})]_{\mathfrak{g}}$
and $\mathscr{D}(x^{[2]})=[x,\mathscr{D}(x)]_{\mathfrak{g}}$, $\forall x,z_{i},z_{j}\in\mathfrak{g}$
that $\delta(d^{k}\omega_{k})=d^{k}(\delta\omega_{k})$.

The proof is complete.
\end{proof}
\begin{theorem} The map $\vartheta^{n}$, $n\geq1$ defined above
is a coboundary operator, i.e., $\vartheta^{n+1}\circ\vartheta^{n}=0$.

\end{theorem}
\begin{proof}
According to Lemma 3.2, we have $\partial^{n}\text{\textopenbullet}\delta=\delta\text{\textopenbullet}\partial^{n}$.
It follows that
\begin{align*}
\vartheta^{2}\vartheta^{1}\varphi_{1} & =\vartheta^{2}((d^{1}\varphi_{1},\omega_{\varphi_{1}}),\delta\varphi_{1})\\
 & =((d^{2}d^{1}\varphi_{1},d^{2}\omega_{\varphi_{1}}),(d^{1}(\delta\varphi_{1})+\delta(d^{1}\varphi_{1}),\omega_{\delta\varphi_{1}}+\delta\omega_{\varphi_{1}}))=0,
\end{align*}
\begin{align*}
 & \vartheta^{3}\vartheta^{2}(\varphi_{2},\omega_{2},\varphi_{1})\\
= & \vartheta^{3}((d^{2}\varphi_{2},d^{2}\omega_{2}),(d^{1}\varphi_{1}+\delta\varphi_{2},\omega_{\varphi_{1}}+\delta\omega_{2}))\\
= & ((d^{3}d^{2}\varphi_{2},d^{3}d^{2}\omega_{2}),(d^{2}(d^{1}\varphi_{1}+\delta\varphi_{2})+\delta(d^{2}\varphi_{2}),d^{2}(\omega_{\varphi_{1}}+\delta\omega_{2})+\delta(d^{2}\omega_{2})))\\
= & ((0,0),((d^{2}\delta+\delta d^{2})\varphi_{2},(d^{2}\delta+\delta d^{2})\omega_{2}))=0
\end{align*}
and
\begin{align*}
 & \vartheta^{n+1}\vartheta^{n}((\varphi_{n},\omega_{n}),(\varphi_{n-1},\omega_{n-1}))\\
= & \vartheta^{n+1}((d^{n}\varphi_{n},d^{n}\omega_{n}),(d^{n-1}\varphi_{n-1}+\delta\varphi_{n},d^{n-1}\omega_{n-1}+\delta\omega_{n}))\\
= & ((d^{n+1}d^{n}\varphi_{n},d^{n+1}d^{n}\omega_{n}),(d^{n}d^{n-1}\varphi_{n-1}+(d^{n}\delta+\delta d^{n})\varphi_{n},d^{n}d^{n-1}\omega_{n-1}+(d^{n}\delta+\delta d^{n})\omega_{n}))\\
= & 0.
\end{align*}

The proof is complete.
\end{proof}
We obtain a cohomology complex $(C_{\mathrm{ResLD}}^{n}(\mathfrak{g};V),\vartheta^{n})$
of the ResLieDer pair $((\mathfrak{g},[2]),\mathscr{D})$, which is
associated to the restricted representation $(V,\rho,\eta_{V})$.
Denote by $Z_{\mathrm{ResLD}}^{n}(\mathfrak{g};V)$ the set of $n$-cocycles
and $B_{\mathrm{ResLD}}^{n}(\mathfrak{g};V)$ the set of $n$-coboundaries.
And the corresponding cohomology group is defined by
\[
H_{\mathrm{ResLD}}^{n}(\mathfrak{g};V)=Z_{\mathrm{ResLD}}^{n}(\mathfrak{g};V)/B_{\mathrm{ResLD}}^{n}(\mathfrak{g};V).
\]

\section{Deformations of a ResLieDer pair}

\subsection{Formal deformations of a ResLieDer pair}

In this subsection, the 1-parameter formal deformation of the ResLieDer
pair $((\mathfrak{g},[2]),\mathscr{D})$ will be studied with the
cohomology associated to the adjoint representation. Note that a 2-cochain
$((\varphi_{2},\omega_{2}),\varphi_{1})\in C_{\mathrm{ResLD}}^{2}(\mathfrak{g};\mathfrak{g})$
is a 2-cocycle if the following equations are satisfied (for all $x,y,z\in\mathfrak{g}$):

\begin{equation}
[x,\varphi_{2}(y,z)]_{\mathfrak{g}}+\varphi_{2}([y,z]_{\mathfrak{g}},x)+\circlearrowleft(x,y,z)=0,
\end{equation}
\begin{equation}
[x,\varphi_{2}(x,y)]_{\mathfrak{g}}+\varphi_{2}(x^{[2]_{\alpha}},y)+\varphi_{2}([x,y]_{\mathfrak{g}},x)+[y,\omega_{2}(x)]_{\mathfrak{g}}=0,
\end{equation}
\begin{equation}
\varphi_{1}([x,y]_{\mathfrak{g}})+[x,\varphi_{1}(y)]_{\mathfrak{g}}+[y,\varphi_{1}(x)]_{\mathfrak{g}}+\varphi_{2}(\mathscr{D}(x),y)+\varphi_{2}(x,\mathscr{D}(y))+\mathscr{D}(\varphi_{2}(x,y))=0,
\end{equation}
\begin{equation}
\varphi_{1}(x^{[2]})+[x,\varphi_{1}(x)]_{\mathfrak{g}}+\mathscr{D}(\omega_{2}(x))+\varphi_{2}(x,\mathscr{D}(x))=0.
\end{equation}

Let $((\mathfrak{g},[2]),\mathscr{D})$ be a ResLieDer pair. Consider
a $t$-parametrized family of operators

\begin{align*}
\mu_{t}(x,y) & =\sum_{i\geq0}\mu_{i}(x,y)t^{i},\\
\sigma_{t}(x) & =\sum_{i\geq0}\sigma_{i}(x)t^{i},\\
\mathscr{D}_{t}(x) & =\sum_{i\geq0}\mathscr{D}_{i}(x)t^{i},
\end{align*}
where $\mu_{i}:\land^{2}\mathfrak{g}\rightarrow\mathfrak{g}$ is a
bilinear map, $\sigma_{i}:\mathfrak{g}\rightarrow\mathfrak{g}$ is
a map and $\mathscr{D}_{i}:\mathfrak{g}\rightarrow\mathfrak{g}$ is
a linear map.

\begin{definition} If all $((\mu_{t},\sigma_{t}),\mathscr{D}_{t})$
endow the $\mathbb{F}[[t]]$-module $\mathfrak{g}[[t]]$ the ResLieDer-pair
structure with $((\mu_{0},\sigma_{0}),\mathscr{D}_{0})=(([\cdot,\cdot]\mathfrak{_{g}},[2]),\mathscr{D})$,
then $((\mu_{t},\sigma_{t}),\mathscr{D}_{t})$ is called a 1-parameter
formal deformation of the ResLieDer pair $((\mathfrak{g},[2]),\mathscr{D})$.

\end{definition}

The pair $((\mu_{t},\sigma_{t}),\mathscr{D}_{t})$ defined as above
is a 1-parameter formal deformation of the ResLieDer pair $((\mathfrak{g},[2]),\mathscr{D})$
if and only if for all $x,y,z\in\mathfrak{g}$ and $a\in\mathbb{F}$,
the following equations hold:
\begin{align*}
\mu_{t}(x,\mu_{t}(y,z))+\circlearrowleft(x,y,z) & =0,\\
\mu_{t}(\sigma_{t}(x),y)+\mu_{t}(x,\mu_{t}(x,y)) & =0,\\
\sigma_{t}(ax)+a^{2}\sigma_{t}(x) & =0,\\
\sigma_{t}(x+y)+\sigma_{t}(x)+\sigma_{t}(y)+\mu_{t}(x,y) & =0,\\
\mathscr{D}_{t}(\mu_{t}(x,y))+\mu_{t}(\mathscr{D}_{t}(x),y)+\mu_{t}(x,\mathscr{D}_{t}(y)) & =0,\\
\mathscr{D}_{t}(\sigma_{t}(x))+\mu_{t}(x,\mathscr{D}_{t}(x)) & =0.
\end{align*}

\begin{lemma} If $((\mu_{t},\sigma_{t}),\mathscr{D}_{t})$ is a 1-parameter
formal deformation of the ResLieDer pair $((\mathfrak{g},[2]),\mathscr{D})$,
then $((\mu_{i},\sigma_{i}),\mathscr{D}_{i})\in C_{\mathrm{ResLD}}^{2}(\mathfrak{g};\mathfrak{g})$,
for all $i\geq0$.

\end{lemma}
\begin{proof}
Since $((\mu_{t},\sigma_{t}),\mathscr{D}_{t})$ is a 1-parameter formal
deformation, we have
\[
\sigma_{t}(x+y)+\sigma_{t}(x)+\sigma_{t}(y)+\mu_{t}(x,y)=0\quad\mathrm{and}\quad\sigma_{t}(ax)+a^{2}\sigma_{t}(x)=0,
\]
for any $x,y\in\mathfrak{g}$ and $a\in\mathbb{F}$. It follows that
\[
\sum_{i\geq0}\sigma_{i}(x+y)t^{i}+\sum_{i\geq0}\sigma_{i}(x)t^{i}+\sum_{i\geq0}\sigma_{i}(y)t^{i}+\sum_{i\geq0}\mu_{i}(x,y)t^{i}=0
\]
and
\[
\sum_{i\geq0}\sigma_{i}(ax)t^{i}=a^{2}\sum_{i\geq0}\sigma_{i}(x)t^{i}=\sum_{i\geq0}(a^{2}\sigma_{i}(x))t^{i}.
\]
 Therefore, for any $i\geq0$, we have
\[
\sigma_{i}(x+y)=\sigma_{i}(x)+\sigma_{i}(y)+\mu_{i}(x,y),\;\sigma_{i}(ax)=a^{2}\sigma_{i}(x).
\]
It implies that $(\mu_{i},\sigma_{i})\in C_{*2}^{2}(\mathfrak{g};\mathfrak{g})$.
Therefore, we have $((\mu_{i},\sigma_{i}),\mathscr{D}_{i})\in C_{\mathrm{ResLD}}^{2}(\mathfrak{g};\mathfrak{g})$.
\end{proof}
\begin{proposition} Let $((\mu_{t},\sigma_{t}),\mathscr{D}_{t})$
be a 1-parameter formal deformation of the ResLieDer pair $((\mathfrak{g},[2]),\mathscr{D}).$
Then $((\mu_{1},\sigma_{1}),\mathscr{D}_{1})\in Z_{\mathrm{ResLD}}^{2}(\mathfrak{g};\mathfrak{g})$,
i.e., it is a 2-cocycle of the ResLieDer pair $((\mathfrak{g},[2]),\mathscr{D})$
with the coefficients in the adjoint representation.

\end{proposition}
\begin{proof}
For any $x,y,z\in\mathfrak{g}$, we have
\begin{align*}
0= & \mu_{t}(x,\mu_{t}(y,z))+\circlearrowleft(x,y,z)=\mu_{t}(x,[y,z]_{\mathfrak{g}}+\sum_{i\geq1}\mu_{i}(y,z)t^{i})+\circlearrowleft(x,y,z)\\
= & \mu_{t}(x,[y,z]_{\mathfrak{g}})+\mu_{t}(x,\sum_{i\geq1}\mu_{i}(y,z))t^{i}+\circlearrowleft(x,y,z)\\
= & [x,[y,z]_{\mathfrak{g}}]_{\mathfrak{g}}+\sum_{j\geq1}\mu_{j}(x,[y,z]_{\mathfrak{g}})t^{j}+\sum_{i\geq1}[x,\mu_{i}(y,z)]_{\mathfrak{g}}t^{i}\\
 & +\sum_{i,j\geq1}\mu_{j}(x,\mu_{i}(y,z))t^{i+j}+\circlearrowleft(x,y,z)\\
= & \sum_{j\geq1}\mu_{j}(x,[y,z]_{\mathfrak{g}})t^{j}+\sum_{i\geq1}[x,\mu_{i}(y,z)]_{\mathfrak{g}}t^{i}+\sum_{i,j\geq1}\mu_{j}(x,\mu_{i}(y,z))t^{i+j}+\circlearrowleft(x,y,z).
\end{align*}
By collecting the coefficients of $t$ in the above equation, we obtian
\[
[x,\mu_{1}(y,z)]_{\mathfrak{g}}+\mu_{1}(x,[y,z]_{\mathfrak{g}})+\circlearrowleft(x,y,z)=0,
\]
which implies that the Eq. (4.1) holds. Moreover, by comparing the
coefficients of $t$ of the following equation
\begin{align*}
0= & \mu_{t}(\sigma_{t}(x),y)+\mu_{t}(x,\mu_{t}(x,y))=\mu_{t}(x^{[2]}+\sum_{i\geq1}\sigma_{i}(x)t^{i},y)+\mu_{t}(x,[x,y]_{\mathfrak{g}}+\sum_{i\geq1}\mu_{i}(x,y)t^{i})\\
= & \mu_{t}(x^{[2]},y)+\sum_{i\geq1}\mu_{t}(\sigma_{i}(x),y)t^{i}+\mu_{t}(x,[x,y]_{\mathfrak{g}})+\sum_{i\geq1}\mu_{t}(x,\mu_{i}(x,y))t^{i}\\
= & [x^{[2]},y]_{\mathfrak{g}}+\sum_{j\geq1}\mu_{j}(x^{[2]},y)t^{j}+\sum_{i\geq1}[\sigma_{i}(x),y]_{\mathfrak{g}}t^{i}+\sum_{i,j\geq1}\mu_{j}(\sigma_{i}(x),y)t^{i+j}\\
 & +[x,[x,y]_{\mathfrak{g}}]_{\mathfrak{g}}+\sum_{j\geq1}\mu_{j}(x,[x,y]_{\mathfrak{g}})t^{j}+\sum_{i\geq1}[x,\mu_{i}(x,y)]_{\mathfrak{g}}t^{i}+\sum_{i,j\geq1}\mu_{j}(x,\mu_{i}(x,y))t^{i+j},
\end{align*}
we obtain
\[
\mu_{1}(x^{[2]},y)+[\sigma_{1}(x),y]_{\mathfrak{g}}+\mu_{1}(x,[x,y]_{\mathfrak{g}})+[x,\mu_{1}(x,y)]_{\mathfrak{g}}=0.
\]
It implies that he Eq. $(4.2)$ is satisfied. Similarly, by expanding
the equation $\mathscr{D}_{t}(\mu_{t}(x,y))+\mu_{t}(\mathscr{D}_{t}(x),y)+\mu_{t}(x,\mathscr{D}_{t}(y))=0$
and comparing the coefficients of $t$, we have
\[
\mathscr{D}_{1}([x,y]_{\mathfrak{g}})+[x,\mathscr{D}_{1}(y)]_{\mathfrak{g}}+[y,\mathscr{D}_{1}(x)]_{\mathfrak{g}}+\mu_{1}(\mathscr{D}(x),y)+\mu_{1}(x,\mathscr{D}(y))+\mathscr{D}(\mu_{1}(x,y))=0.
\]
That is, the Eq. (4.3) is satisfied. It remains to prove that the
Eq. $(4.4)$ is satisfied.

Indeed, for any $x\in\mathfrak{g}$, we have
\begin{align*}
0= & \mathscr{D}_{t}(\sigma_{t}(x))+\mu_{t}(x,\mathscr{D}_{t}(x))=\mathscr{D}_{t}(x^{[2]}+\sum_{i\geq1}\sigma_{i}(x)t^{i})+\mu_{t}(x,\mathscr{D}(x)+\sum_{i\geq1}\mathscr{D}_{i}(x)t^{i})\\
= & \mathscr{D}_{t}(x^{[2]})+\sum_{i\geq1}\mathscr{D}_{t}(\sigma_{i}(x))t^{i}+\mu_{t}(x,\mathscr{D}(x))+\sum_{i\geq1}\mu_{t}(x,\mathscr{D}_{i}(x))t^{i}\\
= & \mathscr{D}(x^{[2]})+\sum_{i\geq1}\mathscr{D}_{i}(x^{[2]})t^{i}+\sum_{i\geq1}\mathscr{D}(\sigma_{i}(x))t^{i}+\sum_{i,j\geq1}\mathscr{D}_{j}(\sigma_{i}(x))t^{i+j}+[x,\mathscr{D}(x)]_{\mathfrak{g}}\\
 & +\sum_{i\geq1}\mu_{i}(x,\mathscr{D}(x))t^{i}+\sum_{i\geq1}[x,\mathscr{D}_{i}(x)]_{\mathfrak{g}}t^{i}+\sum_{i,j\geq1}\mu_{j}(x,\mathscr{D}_{i}(x))t^{i+j}\\
= & \sum_{i\geq1}\mathscr{D}_{i}(x^{[2]})t^{i}+\sum_{i\geq1}\mathscr{D}(\sigma_{i}(x))t^{i}+\sum_{i,j\geq1}\mathscr{D}_{j}(\sigma_{i}(x))t^{i+j}+\sum_{i\geq1}\mu_{i}(x,\mathscr{D}(x))t^{i}\\
 & +\sum_{i\geq1}[x,\mathscr{D}_{i}(x)]_{\mathfrak{g}}t^{i}+\sum_{i,j\geq1}\mu_{j}(x,\mathscr{D}_{i}(x))t^{i+j}.
\end{align*}
By collecting the coefficients of $t$ in the above equation, we obtain
\[
\mathscr{D}_{1}(x^{[2]})+[x,\mathscr{D}_{1}(x)]_{\mathfrak{g}}+\mathscr{D}(\sigma_{1}(x))+\mu_{1}(x,\mathscr{D}(x))=0.
\]
It implies that the Eq. $(4.4)$ is satisfied. Therefore, we have
$((\mu_{1},\sigma_{1}),\mathscr{D}_{1})\in Z_{\mathrm{ResLD}}^{2}(\mathfrak{g};\mathfrak{g})$.
\end{proof}
\begin{definition} The 2-cocycle $((\mu_{1},\sigma_{1}),\mathscr{D}_{1})$
is called the infinitesimal of the 1-parameter formal deformation
$((\mu_{t},\sigma_{t}),\mathscr{D}_{t})$ of the ResLieDer pair $((\mathfrak{g},[2]),\mathscr{D})$.

\end{definition}

\begin{definition} Let $((\tilde{\mu}_{t},\tilde{\sigma}_{t}),\mathscr{\tilde{D}}_{t})$
and $((\mu_{t},\sigma_{t}),\mathscr{D}_{t})$ be 1-parameter formal
deformations of the ResLieDer pair $((\mathfrak{g},[2]),\mathscr{D})$.
A formal isomorphism from $((\tilde{\mu}_{t},\tilde{\sigma}_{t}),\mathscr{\tilde{D}}_{t})$
to $((\mu_{t},\sigma_{t}),\mathscr{D}_{t})$ is a power series $\pi_{t}=\sum_{i\geq0}\pi_{i}t^{i}:\mathfrak{g}[[t]]\rightarrow\mathfrak{g}[[t]]$,
where $\pi_{i}\in C_{*2}^{1}(\mathfrak{g};\mathfrak{g})$ and $\pi_{0}=\mathrm{Id}$,
such that
\begin{align*}
\pi_{t}\text{\textopenbullet}\tilde{\mu}_{t} & =\mu_{t}(\pi_{t}\land\pi_{t}),\\
\pi_{t}\text{\textopenbullet}\tilde{\sigma}_{t} & =\sigma_{t}\text{\textopenbullet}\pi_{t},\\
\pi_{t}\text{\textopenbullet}\tilde{\mathscr{D}}_{t} & =\mathscr{D}_{t}\text{\textopenbullet}\pi_{t}.
\end{align*}

Two 1-parameter formal deformations $((\tilde{\mu}_{t},\tilde{\sigma}_{t}),\mathscr{\tilde{D}}_{t})$
and $((\mu_{t},\sigma_{t}),\mathscr{D}_{t})$ is said to be equivalent
if there exists a formal isomorphism $\pi_{t}:((\tilde{\mu}_{t},\tilde{\sigma}_{t}),\mathscr{\tilde{D}}_{t})\rightarrow((\mu_{t},\sigma_{t}),\mathscr{D}_{t})$.

\end{definition}

\begin{theorem} The infinitesimals of two equivalent 1-parameter
formal deformations of a ResLieDer pair $((\mathfrak{g},[2]),\mathscr{D})$
are in the same cohomology class.

\end{theorem}
\begin{proof}
Let $((\tilde{\mu}_{t},\tilde{\sigma}_{t}),\mathscr{\tilde{D}}_{t})$
and $((\mu_{t},\sigma_{t}),\mathscr{D}_{t})$ be two equivalent 1-parameter
formal deformations of a ResLieDer pair $((\mathfrak{g},[2]),\mathscr{D})$
and $\pi_{t}:((\tilde{\mu}_{t},\tilde{\sigma}_{t}),\mathscr{\tilde{D}}_{t})\rightarrow((\mu_{t},\sigma_{t}),\mathscr{D}_{t})$
is a formal isomorphism. Then, for any $x,y\in\mathfrak{g}$, we have
\begin{align*}
\pi_{t}\text{\textopenbullet}\tilde{\mu}_{t}(x,y) & =\mu_{t}(\pi_{t}(x),\pi_{t}(y)),\\
\pi_{t}\text{\textopenbullet}\tilde{\sigma}_{t}(x) & =\sigma_{t}\text{\textopenbullet}\pi_{t}(x),\\
\pi_{t}\text{\textopenbullet}\tilde{\mathscr{D}}_{t}(x) & =\mathscr{D}_{t}\text{\textopenbullet}\pi_{t}(x).
\end{align*}
By expanding the above equations and collecting the coefficients of
$t$, we obtain
\begin{align*}
\tilde{\mu}_{1}(x,y) & =\mu_{1}(x,y)+\pi_{1}([x,y]_{\mathfrak{g}})+[x,\pi_{1}(y)]\mathfrak{_{g}}+[y,\pi_{1}(x)]\mathfrak{_{g}},\\
\tilde{\sigma}_{1}(x) & =\sigma_{1}(x)+\pi_{1}(x^{[2]})+[x,\pi_{1}(x)]_{\mathfrak{g}},\\
\tilde{\mathscr{D}}_{1}(x) & =\mathscr{D}_{1}(x)+\mathscr{D}(\pi_{1}(x))+\pi_{1}(\mathscr{D}(x)).
\end{align*}
It follows that $((\tilde{\mu}_{t},\tilde{\sigma}_{t}),\mathscr{\tilde{D}}_{t})=((\mu_{t},\sigma_{t}),\mathscr{D}_{t})+\vartheta^{1}(\pi_{1})$.
It implies that
\[
[((\tilde{\mu}_{t},\tilde{\sigma}_{t}),\mathscr{\tilde{D}}_{t})]=[((\mu_{t},\sigma_{t}),\mathscr{D}_{t})]\in H_{\mathrm{ResLD}}^{2}(\mathfrak{g};\mathfrak{g}).
\]

The proof is complete.
\end{proof}
\begin{definition} A 1-parameter formal deformation $((\mu_{t},\sigma_{t}),\mathscr{D}_{t})$
of a ResLieDer pair $((\mathfrak{g},[2]),\mathscr{D})$ is said to
be trivial if it is equivalent to $(([\cdot,\cdot]_{\mathfrak{g}},[2]),\mathscr{D})$,
that is, there exists a formal isomorphism $\pi_{t}=\sum_{i\geq0}\pi_{i}t^{i}:\mathfrak{g}[[t]]\rightarrow\mathfrak{g}[[t]]$,
where $\pi_{i}\in C_{*2}^{1}(\mathfrak{g};\mathfrak{g})$ and $\pi_{0}=\mathrm{Id}$,
such that
\begin{align*}
\pi_{t}\text{\textopenbullet}\mu_{t}(x,y) & =[\pi_{t}(x),\pi_{t}(y)]_{\mathfrak{g}},\\
\pi_{t}\text{\textopenbullet}\sigma_{t}(x) & =(\pi_{t}(x))^{[2]},\\
\pi_{t}\text{\textopenbullet}\mathscr{D}_{t}(x) & =\mathscr{D}\text{\textopenbullet}\pi_{t}(x),
\end{align*}
for all $x,y\in\mathfrak{g}$.

\end{definition}

\begin{definition} A ResLieDer pair $((\mathfrak{g},[2]),\mathscr{D})$
is said to be rigid if every 1-parameter formal deformation of $((\mathfrak{g},[2]),\mathscr{D})$
is trivial.

\end{definition}

\begin{theorem} If $H_{\mathrm{ResLD}}^{2}(\mathfrak{g};\mathfrak{g})=0$,
then the ResLieDer pair $((\mathfrak{g},[2]),\mathscr{D})$ is rigid.

\end{theorem}
\begin{proof}
Let $((\mu_{t},\sigma_{t}),\mathscr{D}_{t})$ be a 1-parameter formal
deformation of $((\mathfrak{g},[2]),\mathscr{D})$. Then, according
to Proposition 4.3, we have $((\mu_{1},\sigma_{1}),\mathscr{D}_{1})\in Z_{\mathrm{ResLD}}^{2}(\mathfrak{g};\mathfrak{g})$.
It follows from $H_{\mathrm{ResLD}}^{2}(\mathfrak{g};\mathfrak{g})=0$
that there exists a 1-cochain $\pi_{1}\in C_{\mathrm{ResLD}}^{1}(\mathfrak{g};\mathfrak{g})$
such that
\begin{equation}
((\mu_{1},\sigma_{1}),\mathscr{D}_{1})=\vartheta^{1}(\pi_{1}).
\end{equation}

Let $\pi_{t}=\mathrm{Id}+\pi_{1}t$, then there exists a 1-parameter
formal deformation $((\tilde{\mu}_{t},\tilde{\sigma}_{t}),\mathscr{\tilde{D}}_{t})$
of the ResLieDer pair $((\mathfrak{g},[2]),\mathscr{D})$ such that
\begin{align*}
\tilde{\mu}_{t}(x,y) & =\pi_{t}^{-1}\text{\textopenbullet}\mu_{t}(\pi_{t}(x),\pi_{t}(y)),\\
\tilde{\sigma}_{t}(x) & =\pi_{t}^{-1}\text{\textopenbullet}\sigma_{t}(\pi_{t}(x)),\\
\mathscr{\tilde{D}}_{t}(x) & =\pi_{t}^{-1}\text{\textopenbullet}\mathscr{D}_{t}(\pi_{t}(x)),
\end{align*}
where $\pi_{t}^{-1}=\sum_{i\geq0}\pi_{1}^{i}t^{i}$. Therefore, the
deformation $((\tilde{\mu}_{t},\tilde{\sigma}_{t}),\mathscr{\tilde{D}}_{t})$
is equivalent to $((\mu_{t},\sigma_{t}),\mathscr{D}_{t})$. Furthermore,
we have

\begin{align*}
\tilde{\mu}_{t}(x,y) & =(\mathrm{Id}+\pi_{1}t+\pi_{1}^{2}t^{2}+\cdots+\pi_{1}^{i}t^{i}+\cdots)(\mu_{t}(x+\pi_{1}(x)t,y+\pi_{1}(y)t)),\\
\tilde{\sigma}_{t}(x) & =(\mathrm{Id}+\pi_{1}t+\pi_{1}^{2}t^{2}+\cdots+\pi_{1}^{i}t^{i}+\cdots)(\sigma_{t}(x+\pi_{1}(x)t)),\\
\mathscr{\tilde{D}}_{t}(x) & =(\mathrm{Id}+\pi_{1}t+\pi_{1}^{2}t^{2}+\cdots+\pi_{1}^{i}t^{i}+\cdots)(\mathscr{D}_{t}(x+\pi_{1}(x)t)).
\end{align*}
It follows that
\begin{align*}
\tilde{\mu}_{t}(x,y) & =[x,y]_{\mathfrak{g}}+(\mu_{1}(x,y)+[x,\pi_{1}(y)]_{\mathfrak{g}}+[\pi_{1}(x),y]_{\mathfrak{g}}+\pi_{1}([x,y]_{\mathfrak{g}}))t+\tilde{\mu}_{2}(x,y)t^{2}+\cdots,\\
\tilde{\sigma}_{t}(x) & =x^{[2]}+(\sigma_{1}(x)+\pi_{1}(x^{[2]})+[x,\pi_{1}(x)]_{\mathfrak{g}})t+\tilde{\sigma}_{2}(x)t^{2}+\cdots,\\
\mathscr{\tilde{D}}_{t}(x) & =\mathscr{D}(x)+(\mathscr{D}_{1}(x)+\mathscr{D}(\pi_{1}(x))+\pi_{1}(\mathscr{D}(x)))t+\mathscr{\tilde{D}}_{2}(x)t^{2}+\cdots.
\end{align*}
According to the Eq. (4.5), we obtain
\begin{align*}
\tilde{\mu}_{t}(x,y) & =[x,y]_{\mathfrak{g}}+\tilde{\mu}_{2}(x,y)t^{2}+\cdots,\\
\tilde{\sigma}_{t}(x) & =x^{[2]}+\tilde{\sigma}_{2}(x)t^{2}+\cdots,\\
\mathscr{\tilde{D}}_{t}(x) & =\mathscr{D}(x)+\mathscr{\tilde{D}}_{2}(x)t^{2}+\cdots.
\end{align*}

By repeating this argument, it can be shown that the deformation $((\mu_{t},\sigma_{t}),\mathscr{D}_{t})$
is equivalent to $(([\cdot,\cdot]_{\mathfrak{g}},[2]),\mathscr{D})$.
It implies that the 1-parameter formal deformation $((\mu_{t},\sigma_{t}),\mathscr{D}_{t})$
is trivial and the ResLieDer pair $((\mathfrak{g},[2]),\mathscr{D})$
is rigid.
\end{proof}

\subsection{Deformations of order $n$ of a ResLieDer pair}

In this subsection, a cohomology class will be defined to study the
deformations of order $n$ of a ResLieDer pair $((\mathfrak{g},[2]),\mathscr{D})$.
It is called a obstruction class of a deformation of order $n$ being
extensible.

\begin{definition} A deformation of order $n$ of a ResLieDer pair
$((\mathfrak{g},[2]),\mathscr{D})$ is a pair $((\mu_{t},\sigma_{t}),\mathscr{D}_{t})$
such that $\mu_{t}=\sum_{i=0}^{n}\mu_{i}t^{i}$, $\sigma_{t}=\sum_{i=0}^{n}\sigma_{i}t^{i}$
and $\mathscr{D}_{t}=\sum_{i=0}^{n}\mathscr{D}_{i}t^{i}$ endow the
$\mathbb{F}[[t]]/(t^{n+1})$-module $\mathfrak{g}[[t]]/(t^{n+1})$
the ResLieDer pair structure with $((\mu_{0},\sigma_{0}),\mathscr{D}_{0})=(([\cdot,\cdot]_{\mathfrak{g}},[2]),\mathscr{D})$.

\end{definition}

\begin{definition} Let $((\mu_{t},\sigma_{t}),\mathscr{D}_{t})$
be a deformation of order $n$ of a ResLieDer pair $((\mathfrak{g},[2]),\mathscr{D})$.
It is said to be extensible if there exists a 2-cochain $((\mu_{n+1},\sigma_{n+1}),\mathscr{D}_{n+1})\in C_{\mathrm{ResLD}}^{2}(\mathfrak{g};\mathfrak{g})$
such that the pair $((\tilde{\mu}_{t},\tilde{\sigma}_{t}),\mathscr{\tilde{D}}_{t})$
with

\[
\tilde{\mu}_{t}=\mu_{t}+\mu_{n+1}t^{n+1},\quad\tilde{\sigma}_{t}=\sigma_{t}+\sigma_{n+1}t^{n+1},\quad\tilde{\mathscr{D}}_{t}=\mathscr{D}_{t}+\mathscr{D}_{n+1}t^{n+1}
\]
is a deformation of order $n+1$ of the ResLieDer pair $((\mathfrak{g},[2]),\mathscr{D})$.

\end{definition}

Let $((\mu_{t},\sigma_{t}),\mathscr{D}_{t})$ be a deformation of
order $n$ of the ResLieDer pair $((\mathfrak{g},[2]),\mathscr{D})$.
We define a 3-cochain $((\mathrm{Ob}_{(\mu_{t},\sigma_{t},\mathscr{D}_{t})\mathrm{I}}^{3},\mathrm{Ob}_{(\mu_{t},\sigma_{t},\mathscr{D}_{t})\mathrm{II}}^{3}),(\mathrm{Ob}_{(\mu_{t},\sigma_{t},\mathscr{D}_{t})\mathrm{I}}^{2},\mathrm{Ob}_{(\mu_{t},\sigma_{t},\mathscr{D}_{t})\mathrm{II}}^{2}))\in C_{\mathrm{ResLD}}^{3}(\mathfrak{g};\mathfrak{g})$
as follows

\begin{align*}
\mathrm{Ob}_{(\mu_{t},\sigma_{t},\mathscr{D}_{t})\mathrm{I}}^{3}(x,y,z) & =\sum_{\substack{i+j=n+1\\
i,j>0
}
}(\mu_{i}(x,\mu_{j}(y,z))+\mu_{i}(y,\mu_{j}(z,x))+\mu_{i}(z,\mu_{j}(x,y))),\\
\mathrm{Ob}_{(\mu_{t},\sigma_{t},\mathscr{D}_{t})\mathrm{II}}^{3}(x,y) & =\sum_{\substack{i+j=n+1\\
i,j>0
}
}(\mu_{i}(x,\sigma_{j}(y))+\mu_{i}(x,\mu_{j}(x,y))),\\
\mathrm{Ob}_{(\mu_{t},\sigma_{t},\mathscr{D}_{t})\mathrm{I}}^{2}(x,y) & =\sum_{\substack{i+j=n+1\\
i,j>0
}
}(\mathscr{D}_{i}(\mu_{j}(x,y))+\mu_{j}(x,\mathscr{D}_{i}(y))+\mu_{j}(y,\mathscr{D}_{i}(x))),\\
\mathrm{Ob}_{(\mu_{t},\sigma_{t},\mathscr{D}_{t})\mathrm{II}}^{2}(x) & =\sum_{\substack{i+j=n+1\\
i,j>0
}
}(\mathscr{D}_{i}(\sigma_{j}(x))+\mu_{j}(x,\mathscr{D}_{i}(x))).
\end{align*}

In the following, the 3-cochain $((\mathrm{Ob}_{(\mu_{t},\sigma_{t},\mathscr{D}_{t})\mathrm{I}}^{3},\mathrm{Ob}_{(\mu_{t},\sigma_{t},\mathscr{D}_{t})\mathrm{II}}^{3}),(\mathrm{Ob}_{(\mu_{t},\sigma_{t},\mathscr{D}_{t})\mathrm{I}}^{2},\mathrm{Ob}_{(\mu_{t},\sigma_{t},\mathscr{D}_{t})\mathrm{II}}^{2}))$
will be simply written as $((\mathrm{Ob}_{\mathrm{I}}^{3},\mathrm{Ob}_{\mathrm{II}}^{3}),(\mathrm{Ob}_{\mathrm{I}}^{2},\mathrm{Ob}_{\mathrm{II}}^{2}))$.

Since $((\mu_{t},\sigma_{t}),\mathscr{D}_{t})$ is a deformation of
order $n$ of the ResLieDer pair $((\mathfrak{g},[2]),\mathscr{D})$,
for $0\leq k\leq n$, we have

\begin{align}
\sum_{\substack{i+j=k\\
i,j\geq0
}
}(\mu_{i}(x,\mu_{j}(y,z))+\mu_{i}(y,\mu_{j}(z,x))+\mu_{i}(z,\mu_{j}(x,y))) & =0,\\
\sum_{\substack{i+j=k\\
i,j\geq0
}
}(\mu_{i}(x,\sigma_{j}(y))+\mu_{i}(x,\mu_{j}(x,y))) & =0,\\
\sum_{\substack{i+j=k\\
i,j\geq0
}
}(\mathscr{D}_{i}(\mu_{j}(x,y))+\mu_{j}(x,\mathscr{D}_{i}(y))+\mu_{j}(y,\mathscr{D}_{i}(x))) & =0,\\
\sum_{\substack{i+j=k\\
i,j\geq0
}
}(\mathscr{D}_{i}(\sigma_{j}(x))+\mu_{j}(x,\mathscr{D}_{i}(x))) & =0.
\end{align}

\begin{definition} Let $((\mu_{t},\sigma_{t}),\mathscr{D}_{t})$
be a deformation of order $n$ of a ResLieDer pair $((\mathfrak{g},[2]),\mathscr{D})$.
The cohomology class $[((\mathrm{Ob_{I}^{3},}\mathrm{Ob_{II}^{3}),(}\mathrm{Ob_{I}^{2},}\mathrm{Ob_{II}^{2}}))]\in H_{\mathrm{ResLD}}^{3}(\mathfrak{g};\mathfrak{g})$
is called the obstruction class of $((\mu_{t},\sigma_{t}),\mathscr{D}_{t})$
being extensible.

\end{definition}

\begin{theorem} Let $((\mu_{t},\sigma_{t}),\mathscr{D}_{t})$ be
a deformation of order $n$ of a ResLieDer pair $((\mathfrak{g},[2]),\mathscr{D})$.
Then the deformation $((\mu_{t},\sigma_{t}),\mathscr{D}_{t})$ is
extensible if and only if the obstruction class $[((\mathrm{Ob_{I}^{3},}\mathrm{Ob_{II}^{3}),(}\mathrm{Ob_{I}^{2},}\mathrm{Ob_{II}^{2}}))]$
is trivial.

\end{theorem}
\begin{proof}
Suppose that $((\mu_{t},\sigma_{t}),\mathscr{D}_{t})$ is extensible.
We extend $((\mu_{t},\sigma_{t}),\mathscr{D}_{t})$ to the deformation
of order $n+1$. Then Eqs. (4.6-4.9) hold for $k=n+1$. It follows
that

\begin{align*}
0 & =\sum_{\substack{i+j=n+1\\
i,j\geq0
}
}(\mu_{i}(x,\mu_{j}(y,z))+\mu_{i}(y,\mu_{j}(z,x))+\mu_{i}(z,\mu_{j}(x,y)))\\
 & =\mathrm{Ob_{I}^{3}}(x,y,z)+\sum_{\substack{i+j=n+1\\
i,j=0
}
}(\mu_{i}(x,\mu_{j}(y,z))+\mu_{i}(y,\mu_{j}(z,x))+\mu_{i}(z,\mu_{j}(x,y)))\\
 & =\mathrm{Ob_{I}^{3}}(x,y,z)+d^{2}\mu_{n+1}(x,y,z),
\end{align*}

\begin{align*}
0 & =\sum_{\substack{i+j=n+1\\
i,j\geq0
}
}(\mu_{i}(x,\sigma_{j}(y))+\mu_{i}(x,\mu_{j}(x,y)))\\
 & =\mathrm{Ob_{II}^{3}}(x,y)+\sum_{\substack{i+j=n+1\\
i,j=0
}
}(\mu_{i}(x,\sigma_{j}(y))+\mu_{i}(x,\mu_{j}(x,y)))\\
 & =\mathrm{Ob_{II}^{3}}(x,y)+d^{2}\sigma_{n+1}(x,y),
\end{align*}

\begin{align*}
0 & =\sum_{\substack{i+j=n+1\\
i,j\geq0
}
}(\mathscr{D}_{i}(\mu_{j}(x,y))+\mu_{j}(x,\mathscr{D}_{i}(y))+\mu_{j}(y,\mathscr{D}_{i}(x)))\\
 & =\mathrm{Ob_{I}^{2}}(x,y)+\sum_{\substack{i+j=n+1\\
i,j=0
}
}(\mathscr{D}_{i}(\mu_{j}(x,y))+\mu_{j}(x,\mathscr{D}_{i}(y))+\mu_{j}(y,\mathscr{D}_{i}(x)))\\
 & =\mathrm{Ob_{I}^{2}}(x,y)+(d^{1}\mathscr{D}_{n+1}(x,y)+\delta\mu_{n+1}(x,y)),
\end{align*}

\begin{align*}
0 & =\sum_{\substack{i+j=n+1\\
i,j\geq0
}
}(\mathscr{D}_{i}(\sigma_{j}(x))+\mu_{j}(x,\mathscr{D}_{i}(x)))\\
 & =\mathrm{Ob_{II}^{2}}(x)+\sum_{\substack{i+j=n+1\\
i,j=0
}
}(\mathscr{D}_{i}(\sigma_{j}(x))+\mu_{j}(x,\mathscr{D}_{i}(x)))\\
 & =\mathrm{Ob_{II}^{2}}(x)+(\sigma_{\mathscr{D}_{n+1}}(x)+\delta\sigma_{n+1}(x)).
\end{align*}
Therefore, we have
\[
\mathrm{Ob_{I}^{3}}=d^{2}\mu_{n+1},\quad\mathrm{Ob_{II}^{3}=d^{2}\sigma_{n+1}},\quad\mathrm{Ob_{I}^{2}}=d^{1}\mathscr{D}_{n+1}+\delta\mu_{n+1},\quad\mathrm{Ob_{II}^{2}}=\sigma_{\mathscr{D}_{n+1}}+\delta\sigma_{n+1}.
\]
It implies that
\[
((\mathrm{Ob_{I}^{3},}\mathrm{Ob_{II}^{3}),(}\mathrm{Ob_{I}^{2},}\mathrm{Ob_{II}^{2}}))=\vartheta^{2}((\mu_{n+1},\sigma_{n+1}),\mathscr{D}_{n+1})\in B_{\mathrm{ResLD}}^{3}(\mathfrak{g};\mathfrak{g}).
\]

Therefore, the obstruction class $[(\mathrm{(Ob_{I}^{3},}\mathrm{Ob_{II}^{3}),(}\mathrm{Ob_{I}^{2},}\mathrm{Ob_{II}^{2}}))]$
is trivial.

Conversely, suppose that the obstruction class $[((\mathrm{Ob_{I}^{3},}\mathrm{Ob_{II}^{3}),(}\mathrm{Ob_{I}^{2},}\mathrm{Ob_{II}^{2}}))]$
is trivial. Then there exists a 2-cochain $((\mu_{n+1},\sigma_{n+1}),\mathscr{D}_{n+1})$
such that
\[
((\mathrm{Ob_{I}^{3},}\mathrm{Ob_{II}^{3}),(}\mathrm{Ob_{I}^{2},}\mathrm{Ob_{II}^{2}}))=\vartheta^{2}((\mu_{n+1},\sigma_{n+1}),\mathscr{D}_{n+1}).
\]

Let $((\tilde{\mu}_{t},\tilde{\sigma}_{t}),\mathscr{\tilde{D}}_{t})=((\mu_{t}+\mu_{n+1}t^{n+1},\sigma_{t}+\sigma_{n+1}t^{n+1}),\mathscr{D}_{t}+\mathscr{D}_{n+1}t^{n+1})$.
Then $((\tilde{\mu}_{t},\tilde{\sigma}_{t}),\mathscr{\tilde{D}}_{t})$
satisfies Eqs. (4.6-4.9) for $0\leq k\leq n+1$ and it is a deformation
of order $n+1$. It implies that $((\mu_{t},\sigma_{t}),\mathscr{D}_{t})$
is extensible.

The proof is complete.
\end{proof}
\begin{corollary} If $H_{\mathrm{ResLD}}^{3}(\mathfrak{g};\mathfrak{g})=0$,
then every 2-cocycle in $C_{\mathrm{ResLD}}^{2}(\mathfrak{g};\mathfrak{g})$
is the infinitesimal of some 1-parameter formal deformation of the
ResLieDer pair $((\mathfrak{g},[2]),\mathscr{D})$.

\end{corollary}

\section{Central extensions of a ResLieDer pair}

In this section, central extensions of a ResLieDer pair will be studied.
It will be shown that the central extensions of a ResLieDer pair $((\mathfrak{g},[2]),\mathscr{D})$
are controlled by the second cohomology of $((\mathfrak{g},[2]),\mathscr{D})$
with values in the trivial representation.

\begin{definition} (see \cite{EF}) Let $(\mathfrak{g},[p]_{\mathfrak{g}})$
be a restricted Lie algebra and $(\mathfrak{h},[p]_{\mathfrak{h}})$
a strongly abelian restricted Lie algebra. A restricted extension
of $\mathfrak{g}$ by $\mathfrak{h}$ is an exact sequence

\[
0\longrightarrow\mathfrak{h}\overset{i}{\longrightarrow}\hat{\mathfrak{g}}\overset{\theta}{\longrightarrow}\mathfrak{\mathfrak{g}}\longrightarrow0
\]
of restricted Lie algebras and restricted Lie algebra morphisms. Such
an extension is called central if $[\mathfrak{h},\hat{\mathfrak{g}}]_{\hat{\mathfrak{g}}}=0$.

\end{definition}

In \cite{EF}, it has been pointed that a restricted extension of
$\mathfrak{g}$ by $\mathfrak{h}$ gives $\mathfrak{h}$ the structure
of a $\mathfrak{g}$-module by the action $x\cdot h=[\hat{x},h]$,
where $\hat{x}\in\hat{\mathfrak{g}}$ is any element satisfying $\theta(\hat{x})=x$.
Moreover, if $[\mathfrak{h},\hat{\mathfrak{g}}]_{\hat{\mathfrak{g}}}=0$,
then $\mathfrak{h}$ is a trivial $\mathfrak{g}$-module.

In the following, a similar extension of a ResLieDer pair $((\mathfrak{g},[2]),\mathscr{D})$
will be introduced. And we can obtain a trivial representation of
$((\mathfrak{g},[2]),\mathscr{D})$ by this extension.

\begin{definition} Let $((\mathfrak{h},[2]_{\mathfrak{h}}),\mathscr{D}_{\mathfrak{h}})$
be a strongly abelian ResLieDer pair and $((\mathfrak{g},[2]_{\mathfrak{g}}),\mathscr{D}_{\mathfrak{g}})$
a ResLieDer pair. An exact sequence of ResLieDer pair morphisms
\begin{align*}
0 & \longrightarrow\mathfrak{h}\overset{i}{\longrightarrow}\hat{\mathfrak{g}}\overset{\theta}{\longrightarrow}\mathfrak{\mathfrak{g}}\longrightarrow0\\
 & \begin{array}{ccccc}
\mathscr{D}_{\mathfrak{h}}\downarrow & \mathscr{D}_{\mathfrak{\hat{g}}}\downarrow & \mathscr{D}_{\mathfrak{g}}\downarrow &  & \begin{array}{ccccc}
\end{array}\end{array}\\
0 & \longrightarrow\mathfrak{h}\underset{i}{\longrightarrow}\hat{\mathfrak{g}}\underset{\theta}{\longrightarrow}\mathfrak{\mathfrak{g}}\longrightarrow0
\end{align*}
is called a central extension of $((\mathfrak{g},[2]_{\mathfrak{g}}),\mathscr{D}_{\mathfrak{g}})$
by $((\mathfrak{h},[2]_{\mathfrak{h}}),\mathscr{D}_{\mathfrak{h}})$
if $[\mathfrak{h},\hat{\mathfrak{g}}]_{\hat{\mathfrak{g}}}=0$.

\end{definition}

If we identify $\mathfrak{h}$ with the corresponding $2$-subalgebra
of $\hat{\mathfrak{g}}$, then $x^{[2]_{\hat{\mathfrak{g}}}}=x^{[2]_{\mathfrak{h}}}$
and $\mathscr{D}_{\hat{\mathfrak{g}}}(x)=\mathscr{D}_{\mathfrak{h}}(x)$,
$\forall x\in\mathfrak{h}$.

\begin{definition} A section of a central extension $((\hat{\mathfrak{g}},[2]_{\mathfrak{\hat{g}}}),\mathscr{D}_{\mathfrak{\hat{g}}})$
of $((\mathfrak{g},[2]_{\mathfrak{g}}),\mathscr{D}_{\mathfrak{g}})$
by $((\mathfrak{h},[2]_{\mathfrak{h}}),\mathscr{D}_{\mathfrak{h}})$
is a linear map $s:\mathfrak{g}\rightarrow\hat{\mathfrak{g}}$ such
that $\theta\text{\textopenbullet}s=\mathrm{Id}$.

\end{definition}

\begin{definition} Let $((\hat{\mathfrak{g}}_{1},[2]_{\mathfrak{\hat{g}}_{1}}),\mathscr{D}_{\mathfrak{\hat{g}}_{1}})$
and $((\hat{\mathfrak{g}}_{2},[2]_{\mathfrak{\hat{g}}_{2}}),\mathscr{D}_{\mathfrak{\hat{g}}_{2}})$
be central extensions of $((\mathfrak{g},[2]_{\mathfrak{g}}),\mathscr{D}_{\mathfrak{g}})$
by $((\mathfrak{h},[2]_{\mathfrak{h}}),\mathscr{D}_{\mathfrak{h}})$.
They are said to be isomorphic if there exists a ResLieDer pair morphism
$\pi:((\hat{\mathfrak{g}}_{1},[2]_{\mathfrak{\hat{g}}_{1}}),\mathscr{D}_{\mathfrak{\hat{g}}_{1}})\rightarrow((\hat{\mathfrak{g}}_{2},[2]_{\mathfrak{\hat{g}}_{2}}),\mathscr{D}_{\mathfrak{\hat{g}}_{2}})$
such that the following diagram commutes

\begin{align*}
0 & \longrightarrow((\mathfrak{h},[2]_{\mathfrak{h}}),\mathscr{D}_{\mathfrak{h}})\overset{i_{1}}{\longrightarrow}((\hat{\mathfrak{g}}_{1},[2]_{\mathfrak{\hat{g}}_{1}}),\mathscr{D}_{\mathfrak{\hat{g}}_{1}})\overset{\theta_{1}}{\longrightarrow}((\mathfrak{g},[2]_{\mathfrak{g}}),\mathscr{D}_{\mathfrak{g}})\longrightarrow0\\
 & \begin{array}{ccccc}
 &  &  &  & \begin{array}{ccccc}
\parallel &  &  &  & \begin{array}{ccc}
\end{array}\begin{array}{cccccc}
\begin{array}{cccc}
 &  &  & \pi\downarrow\end{array} &  &  &  &  & \begin{array}{ccc}
\end{array}\begin{array}{ccccccc}
 &  & \parallel\end{array}\end{array}\end{array}\end{array}\\
0 & \longrightarrow((\mathfrak{h},[2]_{\mathfrak{h}}),\mathscr{D}_{\mathfrak{h}})\underset{i_{2}}{\longrightarrow}((\hat{\mathfrak{g}}_{2},[2]_{\mathfrak{\hat{g}}_{2}}),\mathscr{D}_{\mathfrak{\hat{g}}_{2}})\underset{\theta_{2}}{\longrightarrow}((\mathfrak{g},[2]_{\mathfrak{g}}),\mathscr{D}_{\mathfrak{g}})\longrightarrow0.
\end{align*}

\end{definition}

Let $(\mathfrak{\hat{g}},[2]_{\mathfrak{\hat{g}}},\mathscr{D}_{\hat{\mathfrak{g}}})$
be a central extension of a ResLieDer pair $((\mathfrak{g},[2]_{\mathfrak{g}}),\mathscr{D}_{\mathfrak{g}})$
by a strongly abelian ResLieDer pair $((\mathfrak{h},[2]_{\mathfrak{h}}),\mathscr{D}_{\mathfrak{h}})$
and $s:\mathfrak{g}\rightarrow\hat{\mathfrak{g}}$ a section. We define
a bilinear map $\psi:\mathfrak{g}\land\mathfrak{g}\rightarrow\mathfrak{h}$
and two maps $\varsigma,\tau:\mathfrak{g}\rightarrow\mathfrak{h}$
respectively by

\begin{align*}
\psi(x,y) & =[s(x),s(y)]_{\hat{\mathfrak{g}}}+s([x,y]_{\mathfrak{g}}),\;\forall x,y\in\mathfrak{g},\\
\varsigma(x) & =s(x)^{[2]_{\hat{\mathfrak{g}}}}+s(x^{[2]_{\mathfrak{g}}}),\;\forall x\in\mathfrak{g},\\
\tau(x) & =\mathscr{D}_{\mathfrak{\hat{g}}}(s(x))+s(\mathscr{D}_{\mathfrak{g}}(x)),\;\forall x\in\mathfrak{g}.
\end{align*}

Clearly, $\hat{\mathfrak{g}}$ is isomorphic to $\mathfrak{g}\oplus\mathfrak{h}$
as vector spaces. Moreover, the induced bracket and two maps on $\mathfrak{g}\oplus\mathfrak{h}$
are given by
\begin{align}
[x+h_{1},y+h_{2}]_{\psi} & =[x,y]_{\mathfrak{g}}+\psi(x,y),\;\forall x,y\in\mathfrak{g},h_{1},h_{2}\in\mathfrak{h},\\
(x+h)^{[2]_{\varsigma}} & =x^{[2]_{\mathfrak{g}}}+\varsigma(x),\;\forall x\in\mathfrak{g},h\in\mathfrak{h},\\
\mathscr{D}_{\tau}(x+h) & =\mathscr{D}_{\mathfrak{g}}(x)+\tau(x)+\mathscr{D}_{\mathfrak{h}}(h),\;\forall x\in\mathfrak{g},h\in\mathfrak{h}.
\end{align}

With these notations, we have

\begin{proposition} The pair $((\mathfrak{g}\oplus\mathfrak{h},[2]_{\varsigma}),\mathscr{D}_{\tau})$
with Eqs. (5.1)-(5.3) is a ResLieDer pair if and only if $((\psi,\varsigma),\tau)$
is a 2-cocycle of the ResLieDer pair $((\mathfrak{g},[2]_{\mathfrak{g}}),\mathscr{D}_{\mathfrak{g}})$
with values in the trivial representation $(\mathfrak{h},0,\mathscr{D}_{\mathfrak{h}})$,
that is, $((\psi,\varsigma),\tau)$ satisfies the following equations:
\begin{align}
\psi(x,[y,z]_{\mathfrak{g}})+\psi(y,[z,x]_{\mathfrak{g}})+\psi(z,[x,y]_{\mathfrak{g}}) & =0,\\
\psi(x^{[2]_{\mathfrak{g}}},y)+\psi([x,y]_{\mathfrak{g}},x) & =0,\\
\tau([x,y]_{\mathfrak{g}})+\mathscr{D}_{\mathfrak{h}}(\psi(x,y))+\psi(\mathscr{D}_{\mathfrak{g}}(x),y)+\psi(x,\mathscr{D}_{\mathfrak{g}}(y)) & =0,\\
\tau(x^{[2]_{\mathfrak{g}}})+\mathscr{D}_{\mathfrak{h}}(\varsigma(x))+\psi(x,\mathscr{D}_{\mathfrak{g}}(x)) & =0.
\end{align}

\end{proposition}
\begin{proof}
Suppose that $((\mathfrak{g}\oplus\mathfrak{h},[2]_{\varsigma}),\mathscr{D}_{\tau})$
is a ResLieDer pair. For any $x,y,z\in\mathfrak{g}$ and $h_{1},h_{2},h_{3}\in\mathfrak{h}$,
we have
\begin{align*}
0 & =[x+h_{1},[y+h_{2},z+h_{3}]_{\psi}]_{\psi}+[y+h_{2},[z+h_{3},x+h_{1}]_{\psi}]_{\psi}+[z+h_{3},[x+h_{1},y+h_{2}]_{\psi}]_{\psi}\\
 & =[x+h_{1},[y,z]_{\mathfrak{g}}+\psi(y,z)]_{\psi}+[y+h_{2},[z,x]_{\mathfrak{g}}+\psi(z,x)]_{\psi}+[z+h_{3},[x,y]_{\mathfrak{g}}+\psi(x,y)]_{\psi}\\
 & =[x,[y,z]_{\mathfrak{g}}]_{\mathfrak{g}}+\psi(x,[y,z]_{\mathfrak{g}})+[y,[z,x]_{\mathfrak{g}}]_{\mathfrak{g}}+\psi(y,[z,x]_{\mathfrak{g}})+[z,[x,y]_{\mathfrak{g}}]_{\mathfrak{g}}+\psi(z,[x,y]_{\mathfrak{g}})\\
 & =\psi(x,[y,z]_{\mathfrak{g}})+\psi(y,[z,x]_{\mathfrak{g}})+\psi(z,[x,y]_{\mathfrak{g}})
\end{align*}
and
\begin{align*}
0 & =[(x+h_{1})^{[2]_{\varsigma}},y+h_{2}]_{\psi}+[x+h_{1},[x+h_{1},y+h_{2}]_{\psi}]_{\psi}\\
 & =[x^{[2]_{\mathfrak{g}}}+\varsigma(x),y+h_{2}]_{\psi}+[x+h_{1},[x,y]_{\mathfrak{g}}+\psi(x,y)]_{\psi}\\
 & =[x^{[2]_{\mathfrak{g}}},y]_{\mathfrak{g}}+\psi(x^{[2]_{\mathfrak{g}}},y)+[x,[x,y]_{\mathfrak{g}}]_{\mathfrak{g}}+\psi(x,[x,y]_{\mathfrak{g}})\\
 & =\psi(x^{[2]_{\mathfrak{g}}},y)+\psi(x,[x,y]_{\mathfrak{g}}).
\end{align*}
That is, the Eqs. (5.4) and (5.5) are satisfied. It remains to prove
that the Eqs. (5.6) and (5.7) are satisfied. Indeed, for any $x,y\in\mathfrak{g},h_{1},h_{2}\in\mathfrak{h}$,
we have
\begin{align*}
0= & \mathscr{D}_{\tau}([x+h_{1},y+h_{2}]_{\psi})+[\mathscr{D}_{\tau}(x+h_{1}),y+h_{2}]_{\psi}+[x+h_{1},\mathscr{D}_{\tau}(y+h_{2})]_{\psi}\\
= & \mathscr{D}_{\tau}([x,y]_{\mathfrak{g}}+\psi(x,y))+[\mathscr{D}_{\mathfrak{g}}(x)+\tau(x)+\mathscr{D}_{\mathfrak{h}}(h_{1}),y+h_{2}]_{\psi}\\
 & +[x+h_{1},\mathscr{D}_{\mathfrak{g}}(y)+\tau(y)+\mathscr{D}_{\mathfrak{h}}(h_{2})]_{\psi}\\
= & \mathscr{D}_{\mathfrak{g}}([x,y]_{\mathfrak{g}})+\tau([x,y]_{\mathfrak{g}})+\mathscr{D}_{\mathfrak{h}}(\psi(x,y))+[\mathscr{D}_{\mathfrak{g}}(x),y]_{\mathfrak{g}}+\psi(\mathscr{D}_{\mathfrak{g}}(x),y)\\
 & +[x,\mathscr{D}_{\mathfrak{g}}(y)]_{\mathfrak{g}}+\psi(x,\mathscr{D}_{\mathfrak{g}}(y))\\
= & \tau([x,y]_{\mathfrak{g}})+\mathscr{D}_{\mathfrak{h}}(\psi(x,y))+\psi(\mathscr{D}_{\mathfrak{g}}(x),y)+\psi(x,\mathscr{D}_{\mathfrak{g}}(y))
\end{align*}
Therefore, the Eq. (5.6) is satisfied. Moreover, for any $x\in\mathfrak{g},h\in\mathfrak{h}$,
we have
\begin{align*}
0 & =\mathscr{D}_{\tau}((x+h){}^{[2]_{\varsigma}})+[x+h,\mathscr{D}_{\tau}(x+h)]_{\psi}\\
 & =\mathscr{D}_{\tau}(x{}^{[2]_{\mathfrak{g}}}+\varsigma(x))+[x+h,\mathscr{D}_{\mathfrak{g}}(x)+\tau(x)+\mathscr{D}_{\mathfrak{h}}(h)]_{\psi}\\
 & =\mathscr{D}_{\mathfrak{g}}(x^{[2]_{\mathfrak{g}}})+\tau(x^{[2]_{\mathfrak{g}}})+\mathscr{D}_{\mathfrak{h}}(\varsigma(x))+[x,\mathscr{D}_{\mathfrak{g}}(x)]_{\mathfrak{g}}+\psi(x,\mathscr{D}_{\mathfrak{g}}(x))\\
 & =\tau(x^{[2]_{\mathfrak{g}}})+\mathscr{D}_{\mathfrak{h}}(\varsigma(x))+\psi(x,\mathscr{D}_{\mathfrak{g}}(x)),
\end{align*}
that is, the Eq. (5.7) is satisfied. Therefore, $((\psi,\varsigma),\tau)$
is a 2-cocycle.

Conversely, if $((\psi,\varsigma),\tau)$ is a 2-cocycle of the ResLieDer
pair $((\mathfrak{g},[2]_{\mathfrak{g}}),\mathscr{D}_{\mathfrak{g}})$
with values in the trivial representation $(\mathfrak{h},0,\mathscr{D}_{\mathfrak{h}})$,
then
\[
\vartheta^{2}((\psi,\varsigma),\tau)=((d^{2}\psi,d^{2}\varsigma),(d^{1}\tau+\delta\psi,\varsigma_{\tau}+\delta\varsigma))=0.
\]

A direct computation yields that
\begin{align*}
[x+h_{1},[y+h_{2},z+h_{3}]_{\psi}]_{\psi}+\circlearrowleft(x+h_{1},y+h_{2},z+h_{3}) & =0,\\{}
[(x+h_{1})^{[2]_{\varsigma}},y+h_{2}]_{\psi}+[x+h_{1},[x+h_{1},y+h_{2}]_{\psi}]_{\psi} & =0,\\
\mathscr{D}_{\tau}([x+h_{1},y+h_{2}]_{\psi})+[\mathscr{D}_{\tau}(x+h_{1}),y+h_{2}]_{\psi}+[x+h_{1},\mathscr{D}_{\tau}(y+h_{2})]_{\psi} & =0,\\
\mathscr{D}_{\tau}((x+h){}^{[2]_{\varsigma}})+[x+h,\mathscr{D}_{\tau}(x+h)]_{\psi} & =0,
\end{align*}
for any $x,y,z\in\mathfrak{g}$, $h_{1},h_{2},h_{3}\in\mathfrak{h}$.
Therefore, $((\mathfrak{g}\oplus\mathfrak{h},[2]_{\varsigma}),\mathscr{D}_{\tau})$
defined by Eqs. (5.1)-(5.3) is a ResLieDer pair.

The proof is complete.
\end{proof}
\begin{lemma} Let $((\mathfrak{\hat{g}},[2]_{\hat{\mathfrak{g}}}),\mathscr{D}_{\mathfrak{\hat{g}}})$
be a central extension of a ResLieDer pair $((\mathfrak{g},[2]_{\mathfrak{g}}),\mathscr{D}_{\mathfrak{g}})$
by a strongly abelian ResLieDer pair $((\mathfrak{h},[2]_{\mathfrak{h}}),\mathscr{D}_{\mathfrak{h}})$
and $s:\mathfrak{g}\rightarrow\hat{\mathfrak{g}}$ a section. If $((\psi,\varsigma),\tau)$
is a 2-cocycle of $((\mathfrak{g},[2]_{\mathfrak{g}}),\mathscr{D}_{\mathfrak{g}})$
with values in the trivial representation $(\mathfrak{h},0,\mathscr{D}_{\mathfrak{h}})$,
then the cohomology class of $((\psi,\varsigma),\tau)$ does not depend
on the choice of sections.

\end{lemma}
\begin{proof}
Let $s_{1},s_{2}:\mathfrak{g}\rightarrow\hat{\mathfrak{g}}$ be sections.
The corresponding 2-cocycles are $((\psi_{1},\varsigma_{1}),\tau_{1})$
and $((\psi_{2},\varsigma_{2}),\tau_{2})$, respectively. We define
a linear map $\kappa:\mathfrak{g}\rightarrow\mathfrak{h}$ by $\kappa=s_{1}-s_{2}$.
Then we have
\begin{align*}
\psi_{1}(x,y) & =[s_{1}(x),s_{1}(y)]_{\hat{\mathfrak{g}}}+s_{1}([x,y]_{\mathfrak{g}})\\
 & =[\kappa(x)+s_{2}(x),\kappa(y)+s_{2}(y)]_{\hat{\mathfrak{g}}}+\kappa([x,y]_{\mathfrak{g}})+s_{2}([x,y]_{\mathfrak{g}})\\
 & =[s_{2}(x),s_{2}(y)]_{\hat{\mathfrak{g}}}+s_{2}([x,y]_{\mathfrak{g}})+\kappa([x,y]_{\mathfrak{g}})\\
 & =\psi_{2}(x,y)+d^{1}\kappa(x,y),
\end{align*}

\begin{align*}
\varsigma_{1}(x) & =s_{1}(x)^{[2]_{\hat{\mathfrak{g}}}}+s_{1}(x^{[2]_{\mathfrak{g}}})=(\kappa(x)+s_{2}(x))^{[2]_{\hat{\mathfrak{g}}}}+\kappa(x^{[2]_{\mathfrak{g}}})+s_{2}(x^{[2]_{\mathfrak{g}}})\\
 & =s_{2}(x)^{[2]_{\hat{\mathfrak{g}}}}+s_{2}(x^{[2]_{\mathfrak{g}}})+\kappa(x^{[2]_{\mathfrak{g}}})=\varsigma_{2}(x)+\omega_{\kappa}(x)
\end{align*}
and
\begin{align*}
\tau_{1}(x) & =\mathscr{D}_{\hat{\mathfrak{g}}}(s_{1}(x))+s_{1}(\mathscr{D}_{\mathfrak{g}}(x))=\mathscr{D}_{\hat{\mathfrak{g}}}(\kappa(x)+s_{2}(x))+\kappa(\mathscr{D}_{\mathfrak{g}}(x))+s_{2}(\mathscr{D}_{\mathfrak{g}}(x))\\
 & =\mathscr{D}_{\mathfrak{h}}(\kappa(x))+\mathscr{D}_{\hat{\mathfrak{g}}}(s_{2}(x))+\kappa(\mathscr{D}_{\mathfrak{g}}(x))+s_{2}(\mathscr{D}_{\mathfrak{g}}(x))=\tau_{2}(x)+\delta\kappa(x).
\end{align*}
It follows that $((\psi_{1},\varsigma_{1}),\tau_{1})-((\psi_{2},\varsigma_{2}),\tau_{2})=((d^{1}\kappa,\omega_{\kappa}),\delta\kappa)=\vartheta^{1}\kappa\in B_{\mathrm{ResLD}}^{2}(\mathfrak{g};\mathfrak{h})$.
Therefore, $((\psi_{1},\varsigma_{1}),\tau_{1})$ and $((\psi_{2},\varsigma_{2}),\tau_{2})$
are in the same cohomology class. It implies that the cohomology class
of $((\psi,\varsigma),\tau)$ does not depend on the choice of sections.
\end{proof}
\begin{theorem} Let $((\mathfrak{g},[2]_{\mathfrak{g}}),\mathscr{D}_{\mathfrak{g}})$
be a ResLieDer pair and $((\mathfrak{h},[2]_{\mathfrak{h}}),\mathscr{D}_{\mathfrak{h}})$
a strongly abelian ResLieDer pair. The central extensions of $((\mathfrak{g},[2]_{\mathfrak{g}}),\mathscr{D}_{\mathfrak{g}})$
by $((\mathfrak{h},[2]_{\mathfrak{h}}),\mathscr{D}_{\mathfrak{h}})$
are classified by the second cohomology group $H_{\mathrm{ResLD}}^{2}(\mathfrak{g};\mathfrak{h})$
of the ResLieDer pair $((\mathfrak{g},[2]_{\mathfrak{g}}),\mathscr{D}_{\mathfrak{g}})$
with values in the trivial representation $(\mathfrak{h},0,\mathscr{D}_{\mathfrak{h}})$.

\end{theorem}
\begin{proof}
Let $((\mathfrak{\hat{g}},[2]_{\mathfrak{\hat{g}}}),\mathscr{D}_{\mathfrak{\hat{g}}})$
be a central extension of a ResLieDer pair $((\mathfrak{g},[2]_{\mathfrak{g}}),\mathscr{D}_{\mathfrak{g}})$
by a strongly abelian ResLieDer pair $((\mathfrak{h},[2]_{\mathfrak{h}}),\mathscr{D}_{\mathfrak{h}})$,
$s:\mathfrak{g}\rightarrow\hat{\mathfrak{g}}$ a section and $((\psi,\varsigma),\tau)$
the corresponding 2-cocycle. According to Lemma 5.6, the cohomology
class of $((\psi,\varsigma),\tau)$ does not depend on the choice
of sections.

Next, we are going to prove that isomorphic central extensions give
rise to the same element in $H_{\mathrm{ResLD}}^{2}(\mathfrak{g};\mathfrak{h})$.
Let $((\mathfrak{\hat{g}}_{1},[2]_{\mathfrak{\hat{g}}_{1}}),\mathscr{D}_{\mathfrak{\hat{g}}_{1}})$
and $((\mathfrak{\hat{g}}_{2},[2]_{\mathfrak{\hat{g}}_{2}}),\mathscr{D}_{\mathfrak{\hat{g}}_{2}})$
be two isomorphic central extensions of $((\mathfrak{g},[2]_{\mathfrak{g}}),\mathscr{D}_{\mathfrak{g}})$
by $((\mathfrak{h},[2]_{\mathfrak{h}}),\mathscr{D}_{\mathfrak{h}})$
and $\pi:((\mathfrak{\hat{g}}_{1},[2]_{\mathfrak{\hat{g}}_{1}}),\mathscr{D}_{\mathfrak{\hat{g}}_{1}})\rightarrow((\mathfrak{\hat{g}}_{2},[2]_{\mathfrak{\hat{g}}_{2}}),\mathscr{D}_{\mathfrak{\hat{g}}_{2}})$
a ResLieDer pair morphism such that we have the commutative diagram
in Definition 5.4.

Let $s_{1}:\mathfrak{g}\rightarrow\hat{\mathfrak{g}}_{1}$ be a section
of the central extension $((\mathfrak{\hat{g}}_{1},[2]_{\mathfrak{\hat{g}}_{1}}),\mathscr{D}_{\mathfrak{\hat{g}}_{1}})$.
It follows from $\theta_{2}\text{\textopenbullet}\pi=\theta_{1}$
that
\[
\theta_{2}\text{\textopenbullet}(\pi\text{\textopenbullet}s_{1})=(\theta_{2}\text{\textopenbullet}\pi)\text{\textopenbullet}s_{1}=\theta_{1}\text{\textopenbullet}s_{1}=\mathrm{Id}.
\]

It implies that $\pi\text{\textopenbullet}s_{1}$ is a section of
the central extension $((\mathfrak{\hat{g}}_{2},[2]_{\mathfrak{\hat{g}}_{2}}),\mathscr{D}_{\mathfrak{\hat{g}}_{2}})$.
Let $s_{2}=\pi\text{\textopenbullet}s_{1}$. Then we obtain two corresponding
2-cocycles $((\psi_{1},\varsigma_{1}),\tau_{1})$ and $((\psi_{2},\varsigma_{2}),\tau_{2})$.
Since $\pi$ is a ResLieDer pair morphism and $\pi|_{\mathfrak{h}}=\mathrm{Id}_{\mathfrak{h}}$,
we have
\begin{align*}
\psi_{2}(x,y) & =[s_{2}(x),s_{2}(y)]_{\hat{\mathfrak{g}}}+s_{2}([x,y]_{\mathfrak{g}})=[\pi(s_{1}(x)),\pi(s_{1}(y))]_{\hat{\mathfrak{g}}}+\pi(s_{1}([x,y]_{\mathfrak{g}}))\\
 & =\pi([s_{1}(x),s_{1}(y)]_{\hat{\mathfrak{g}}}+s_{1}([x,y]_{\mathfrak{g}}))=\pi|_{\mathfrak{h}}(\psi_{1}(x,y))=\psi_{1}(x,y),
\end{align*}

\begin{align*}
\varsigma_{2}(x) & =s_{2}(x)^{[2]_{\hat{\mathfrak{g}}_{2}}}+s_{2}(x^{[2]_{\mathfrak{g}}})=\pi(s_{1}(x))^{[2]_{\hat{\mathfrak{g}}_{2}}}+\pi(s_{1}(x^{[2]_{\mathfrak{g}}}))\\
 & =\pi(s_{1}(x)^{[2]_{\hat{\mathfrak{g}}_{1}}})+\pi(s_{1}(x^{[2]_{\mathfrak{g}}}))=\pi(s_{1}(x)^{[2]_{\hat{\mathfrak{g}}_{1}}}+s_{1}(x^{[2]_{\mathfrak{g}}}))\\
 & =s_{1}(x)^{[2]_{\hat{\mathfrak{g}}_{1}}}+s_{1}(x^{[2]_{\mathfrak{g}}})=\varsigma_{1}(x)
\end{align*}
and
\begin{align*}
\tau_{2}(x) & =\mathscr{D}_{\hat{\mathfrak{g}}_{2}}(s_{2}(x))+s_{2}(\mathscr{D}_{\mathfrak{g}}(x))=\mathscr{D}_{\hat{\mathfrak{g}}_{2}}(\pi(s_{1}(x)))+\pi(s_{1}(\mathscr{D}_{\mathfrak{g}}(x)))\\
 & =\pi(\mathscr{D}_{\hat{\mathfrak{g}}_{1}}(s_{1}(x)))+\pi(s_{1}(\mathscr{D}_{\mathfrak{g}}(x)))=\pi(\mathscr{D}_{\hat{\mathfrak{g}}_{1}}(s_{1}(x))+s_{1}(\mathscr{D}_{\mathfrak{g}}(x)))\\
 & =\mathscr{D}_{\hat{\mathfrak{g}}_{1}}(s_{1}(x))+s_{1}(\mathscr{D}_{\mathfrak{g}}(x))=\tau_{1}(x).
\end{align*}
It follows that $((\psi_{2},\varsigma_{2}),\tau_{2})=((\psi_{1},\varsigma_{1}),\tau_{1})$.
It implies that the isomorphic central extensions give rise to the
same element in $H_{\mathrm{ResLD}}^{2}(\mathfrak{g};\mathfrak{h})$.

Conversely, given two 2-cocycles $((\psi_{1},\varsigma_{1}),\tau_{1})$
and $((\psi_{2},\varsigma_{2}),\tau_{2})$, we can construct two central
extensions $((\mathfrak{g}\oplus\mathfrak{h},[2]_{\varsigma_{1}}),\mathscr{D}_{\tau_{1}})$
and $((\mathfrak{g}\oplus\mathfrak{h},[2]_{\varsigma_{2}}),\mathscr{D}_{\tau_{2}})$,
as in the Eqs. (5.1)-(5.3). If they represent the same cohomological
class, that is, there exists $\nu:\mathfrak{g}\rightarrow\mathfrak{h}$
such that
\[
((\psi_{1},\varsigma_{1}),\tau_{1})=((\psi_{2},\varsigma_{2}),\tau_{2})+\vartheta^{1}\nu,
\]
we define $\varrho:\mathfrak{g}\oplus\mathfrak{h}\rightarrow\mathfrak{g}\oplus\mathfrak{h}$
by
\[
\varrho(x+h)=x+\nu(x)+h.
\]
Then, it can be deduced that $\varrho$ is an isomorphism between
central extensions.

The proof is complete.
\end{proof}

\section{Extensions of a pair of restricted derivations}

\begin{definition} Let $0\longrightarrow\mathfrak{h}\overset{i}{\longrightarrow}\hat{\mathfrak{g}}\overset{\theta}{\longrightarrow}\mathfrak{\mathfrak{g}}\longrightarrow0$
be a central extension of restricted Lie algebras. A pair of restricted
derivations $(\mathscr{D}_{\mathfrak{h}},\mathscr{D}_{\mathfrak{g}})\in\mathrm{Der}^{2}(\mathfrak{h})\times\mathrm{Der}^{2}(\mathfrak{g})$
is said to be extensible if there exists a restricted derivation $\mathscr{D}_{\hat{\mathfrak{g}}}\in\mathrm{Der}^{2}(\hat{\mathfrak{g}})$
such that we have the following exact sequence of ResLieDer pair morphisms

\begin{align*}
0 & \longrightarrow\mathfrak{h}\overset{i}{\longrightarrow}\hat{\mathfrak{g}}\overset{\theta}{\longrightarrow}\mathfrak{\mathfrak{g}}\longrightarrow0\\
 & \begin{array}{ccccc}
\mathscr{D}_{\mathfrak{h}}\downarrow & \mathscr{D}_{\mathfrak{\hat{g}}}\downarrow & \mathscr{D}_{\mathfrak{g}}\downarrow &  & \begin{array}{ccccc}
\end{array}\end{array}\\
0 & \longrightarrow\mathfrak{h}\underset{i}{\longrightarrow}\hat{\mathfrak{g}}\underset{\theta}{\longrightarrow}\mathfrak{\mathfrak{g}}\longrightarrow0.
\end{align*}

Equivalently, $((\hat{\mathfrak{g}},[2]_{\hat{\mathfrak{g}}}),\mathscr{D}_{\mathfrak{\hat{g}}})$
is a central extension of $((\mathfrak{g},[2]_{\mathfrak{g}}),\mathscr{D}_{\mathfrak{g}})$
by $((\mathfrak{h},[2]_{\mathfrak{h}}),\mathscr{D}_{\mathfrak{h}})$.

\end{definition}

Let $s:\mathfrak{g}\rightarrow\hat{\mathfrak{g}}$ be any section
of the central extension $0\longrightarrow\mathfrak{h}\overset{i}{\longrightarrow}\hat{\mathfrak{g}}\overset{\theta}{\longrightarrow}\mathfrak{\mathfrak{g}}\longrightarrow0$.
Then any element $\hat{x}\in\hat{\mathfrak{g}}$ can be uniquely written
as $\hat{x}=s(x)+h$, for some $x\in\mathfrak{g}$ and $h\in\mathfrak{h}$.
Define $\psi:\mathfrak{g}\land\mathfrak{g}\rightarrow\mathfrak{h}$
and $\varsigma:\mathfrak{g}\rightarrow\mathfrak{h}$ respectively
by
\begin{align}
\psi(x,y) & =[s(x),s(y)]_{\hat{\mathfrak{g}}}+s([x,y]_{\mathfrak{g}}),\\
\varsigma(x) & =s(x)^{[2]_{\hat{\mathfrak{g}}}}+s(x^{[2]_{\mathfrak{g}}}).
\end{align}

For any pair $(\mathscr{D}_{\mathfrak{h}},\mathscr{D}_{\mathfrak{g}})\in\mathrm{Der}^{2}(\mathfrak{h})\times\mathrm{Der}^{2}(\mathfrak{g})$,
define a 2-cochain $(\mathrm{Ob}_{(\mathscr{D}_{\mathfrak{h}},\mathscr{D}_{\mathfrak{g}})\mathrm{I}}^{\hat{\mathfrak{g}}},\mathrm{Ob}_{(\mathscr{D}_{\mathfrak{h}},\mathscr{D}_{\mathfrak{g}})\mathrm{II}}^{\hat{\mathfrak{g}}})\in C_{*2}^{2}(\mathfrak{g};\mathfrak{h})$
as follows:
\begin{align}
\mathrm{Ob}_{(\mathscr{D}_{\mathfrak{h}},\mathscr{D}_{\mathfrak{g}})\mathrm{I}}^{\hat{\mathfrak{g}}}(x,y) & =\mathscr{D}_{\mathfrak{h}}(\psi(x,y))+\psi(\mathscr{D}_{\mathfrak{g}}(x),y)+\psi(x,\mathscr{D}_{\mathfrak{g}}(y)),\;\forall x,y\in\mathfrak{g},\\
\mathrm{Ob}_{(\mathscr{D}_{\mathfrak{h}},\mathscr{D}_{\mathfrak{g}})\mathrm{II}}^{\hat{\mathfrak{g}}}(x) & =\mathscr{D}_{\mathfrak{h}}(\varsigma(x))+\psi(x,\mathscr{D}_{\mathfrak{g}}(x)),\;\forall x\in\mathfrak{g}.
\end{align}

\begin{proposition} Let $0\longrightarrow\mathfrak{h}\overset{i}{\longrightarrow}\hat{\mathfrak{g}}\overset{\theta}{\longrightarrow}\mathfrak{\mathfrak{g}}\longrightarrow0$
be a central extension of restricted Lie algebras. For any pair of
restricted derivations $(\mathscr{D}_{\mathfrak{h}},\mathscr{D}_{\mathfrak{g}})\in\mathrm{Der}^{2}(\mathfrak{h})\times\mathrm{Der}^{2}(\mathfrak{g})$,
the 2-cochain $(\mathrm{Ob}_{(\mathscr{D}_{\mathfrak{h}},\mathscr{D}_{\mathfrak{g}})\mathrm{I}}^{\hat{\mathfrak{g}}},\mathrm{Ob}_{(\mathscr{D}_{\mathfrak{h}},\mathscr{D}_{\mathfrak{g}})\mathrm{II}}^{\hat{\mathfrak{g}}})\in C_{*2}^{2}(\mathfrak{g};\mathfrak{h})$
defined by the Eqs. (6.3) and (6.4) is a 2-cocycle of the restricted
Lie algebra $(\mathfrak{g},[2]_{\mathfrak{g}})$ with values in the
trivial representation $(\mathfrak{h},0)$.

Moreover, the cohomology class $[(\mathrm{Ob}_{(\mathscr{D}_{\mathfrak{h}},\mathscr{D}_{\mathfrak{g}})\mathrm{I}}^{\hat{\mathfrak{g}}},\mathrm{Ob}_{(\mathscr{D}_{\mathfrak{h}},\mathscr{D}_{\mathfrak{g}})\mathrm{II}}^{\hat{\mathfrak{g}}})]\in H_{*2}^{2}(\mathfrak{g};\mathfrak{h})$
does not depend on the choice of sections.

\end{proposition}
\begin{proof}
Let $s:\mathfrak{g}\rightarrow\hat{\mathfrak{g}}$ be a section of
the central extension of restricted Lie algebras $0\longrightarrow\mathfrak{h}\overset{i}{\longrightarrow}\hat{\mathfrak{g}}\overset{\theta}{\longrightarrow}\mathfrak{\mathfrak{g}}\longrightarrow0$.
According to Proposition 5.5, the pair $(\psi,\varsigma)$ is a 2-cocycle
of the restricted Lie algebra $(\mathfrak{g},[2]_{\mathfrak{g}})$
with values in the trivial representation $(\mathfrak{h},0)$, which
implies that
\[
\psi(x,[y,z]_{\mathfrak{g}})+\psi(y,[z,x]_{\mathfrak{g}})+\psi(z,[x,y]_{\mathfrak{g}})=0,
\]
\[
\psi(x^{[2]_{\mathfrak{g}}},y)+\psi([x,y]_{\mathfrak{g}},x)=0.
\]
Furthermore, for any $x,y,z\in\mathfrak{g}$, we have

\begin{align*}
 & (d^{2}\mathrm{Ob}_{(\mathscr{D}_{\mathfrak{h}},\mathscr{D}_{\mathfrak{g}})\mathrm{I}}^{\hat{\mathfrak{g}}})(x,y,z)\\
= & \mathrm{Ob}_{(\mathscr{D}_{\mathfrak{h}},\mathscr{D}_{\mathfrak{g}})\mathrm{I}}^{\hat{\mathfrak{g}}}(x,[y,z]_{\mathfrak{g}})+\mathrm{Ob}_{(\mathscr{D}_{\mathfrak{h}},\mathscr{D}_{\mathfrak{g}})\mathrm{I}}^{\hat{\mathfrak{g}}}(y,[z,x]_{\mathfrak{g}})+\mathrm{Ob}_{(\mathscr{D}_{\mathfrak{h}},\mathscr{D}_{\mathfrak{g}})\mathrm{I}}^{\hat{\mathfrak{g}}}(z,[x,y]_{\mathfrak{g}})\\
= & \mathscr{D}_{\mathfrak{h}}(\psi(x,[y,z]_{\mathfrak{g}}))+\psi(\mathscr{D}_{\mathfrak{g}}(x),[y,z]_{\mathfrak{g}})+\psi(x,\mathscr{D}_{\mathfrak{g}}([y,z]_{\mathfrak{g}}))\\
 & +\mathscr{D}_{\mathfrak{h}}(\psi(y,[z,x]_{\mathfrak{g}}))+\psi(\mathscr{D}_{\mathfrak{g}}(y),[z,x]_{\mathfrak{g}})+\psi(y,\mathscr{D}_{\mathfrak{g}}([z,x]_{\mathfrak{g}}))\\
 & +\mathscr{D}_{\mathfrak{h}}(\psi(z,[x,y]_{\mathfrak{g}}))+\psi(\mathscr{D}_{\mathfrak{g}}(z),[x,y]_{\mathfrak{g}})+\psi(z,\mathscr{D}_{\mathfrak{g}}([x,y]_{\mathfrak{g}}))\\
= & \mathscr{D}_{\mathfrak{h}}(\psi(x,[y,z]_{\mathfrak{g}})+\psi(y,[z,x]_{\mathfrak{g}})+\psi(y,[z,x]_{\mathfrak{g}}))\\
 & +(\psi(\mathscr{D}_{\mathfrak{g}}(x),[y,z]_{\mathfrak{g}})+\psi(y,[z,\mathscr{D}_{\mathfrak{g}}(x)]_{\mathfrak{g}})+\psi(z,[\mathscr{D}_{\mathfrak{g}}(x),y]_{\mathfrak{g}}))\\
 & +(\psi(\mathscr{D}_{\mathfrak{g}}(y),[z,x]_{\mathfrak{g}})+\psi(z,[x,\mathscr{D}_{\mathfrak{g}}(y)]_{\mathfrak{g}})+\psi(x,[\mathscr{D}_{\mathfrak{g}}(y),z]_{\mathfrak{g}}))\\
 & +(\psi(\mathscr{D}_{\mathfrak{g}}(z),[x,y]_{\mathfrak{g}})+\psi(x,[y,\mathscr{D}_{\mathfrak{g}}(z)]_{\mathfrak{g}})+\psi(y,[\mathscr{D}_{\mathfrak{g}}(z),x]_{\mathfrak{g}}))\\
= & 0
\end{align*}
and
\begin{align*}
 & (d^{2}\mathrm{Ob}_{(\mathscr{D}_{\mathfrak{h}},\mathscr{D}_{\mathfrak{g}})\mathrm{II}}^{\hat{\mathfrak{g}}})(x,y)\\
= & \mathrm{Ob}_{(\mathscr{D}_{\mathfrak{h}},\mathscr{D}_{\mathfrak{g}})\mathrm{I}}^{\hat{\mathfrak{g}}}(x^{[2]_{\mathfrak{g}}},y)+\mathrm{Ob}_{(\mathscr{D}_{\mathfrak{h}},\mathscr{D}_{\mathfrak{g}})\mathrm{I}}^{\hat{\mathfrak{g}}}([x,y]_{\mathfrak{g}},x)\\
= & \mathscr{D}_{\mathfrak{h}}(\psi(x^{[2]_{\mathfrak{g}}},y))+\psi(\mathscr{D}_{\mathfrak{g}}(x^{[2]_{\mathfrak{g}}}),y)+\psi(x^{[2]_{\mathfrak{g}}},\mathscr{D}_{\mathfrak{g}}(y))\\
 & +\mathscr{D}_{\mathfrak{h}}(\psi([x,y]_{\mathfrak{g}},x))+\psi(\mathscr{D}_{\mathfrak{g}}([x,y]_{\mathfrak{g}}),x)+\psi([x,y]_{\mathfrak{g}},\mathscr{D}_{\mathfrak{g}}(x))\\
= & \mathscr{D}_{\mathfrak{h}}(\psi(x^{[2]_{\mathfrak{g}}},y)+\psi([x,y]_{\mathfrak{g}},x))+(\psi(x^{[2]_{\mathfrak{g}}},\mathscr{D}_{\mathfrak{g}}(y))+\psi([x,\mathscr{D}_{\mathfrak{g}}(y)]_{\mathfrak{g}},x))\\
 & +\psi([x,\mathscr{D}_{\mathfrak{g}}(x)]_{\mathfrak{g}},y)+\psi([\mathscr{D}_{\mathfrak{g}}(x),y]_{\mathfrak{g}},x)+\psi([y,x]_{\mathfrak{g}},\mathscr{D}_{\mathfrak{g}}(x))\\
= & 0.
\end{align*}
It follows that $(\mathrm{Ob}_{(\mathscr{D}_{\mathfrak{h}},\mathscr{D}_{\mathfrak{g}})\mathrm{I}}^{\hat{\mathfrak{g}}},\mathrm{Ob}_{(\mathscr{D}_{\mathfrak{h}},\mathscr{D}_{\mathfrak{g}})\mathrm{II}}^{\hat{\mathfrak{g}}})\in Z_{*2}^{2}(\mathfrak{g};\mathfrak{h})$.

Let $s_{1},s_{2}:\mathfrak{g}\rightarrow\hat{\mathfrak{g}}$ be different
sections. And the corresponding 2-cocycles are $(\mathrm{Ob}_{(\mathscr{D}_{\mathfrak{h}},\mathscr{D}_{\mathfrak{g}})\mathrm{I}}^{\hat{\mathfrak{g}},1},\mathrm{Ob}_{(\mathscr{D}_{\mathfrak{h}},\mathscr{D}_{\mathfrak{g}})\mathrm{II}}^{\hat{\mathfrak{g}},1})$
and $(\mathrm{Ob}_{(\mathscr{D}_{\mathfrak{h}},\mathscr{D}_{\mathfrak{g}})\mathrm{I}}^{\hat{\mathfrak{g}},2},\mathrm{Ob}_{(\mathscr{D}_{\mathfrak{h}},\mathscr{D}_{\mathfrak{g}})\mathrm{II}}^{\hat{\mathfrak{g}},2})$,
respectively. Define $\kappa:\mathfrak{g}\rightarrow\mathfrak{h}$
by $\kappa=s_{1}-s_{2}$. Then we have
\begin{align*}
\psi_{1}(x,y) & =[s_{1}(x),s_{1}(y)]_{\hat{\mathfrak{g}}}+s_{1}([x,y]_{\mathfrak{g}})\\
 & =[\kappa(x)+s_{2}(x),\kappa(y)+s_{2}(y)]_{\hat{\mathfrak{g}}}+\kappa([x,y]_{\mathfrak{g}})+s_{2}([x,y]_{\mathfrak{g}})\\
 & =\psi_{2}(x,y)+\kappa([x,y]_{\mathfrak{g}})
\end{align*}
and
\begin{align*}
\varsigma_{1}(x) & =s_{1}(x)^{[2]_{\hat{\mathfrak{g}}}}+s_{1}(x^{[2]_{\mathfrak{g}}})=(\kappa(x)+s_{2}(x))^{[2]_{\hat{\mathfrak{g}}}}+\kappa(x^{[2]_{\mathfrak{g}}})+s_{1}(x^{[2]_{\mathfrak{g}}})\\
 & =s_{2}(x)^{[2]_{\hat{\mathfrak{g}}}}+s_{2}(x^{[2]_{\mathfrak{g}}})+\kappa(x^{[2]_{\mathfrak{g}}})=\varsigma_{2}(x)+\kappa(x^{[2]_{\mathfrak{g}}}).
\end{align*}
It follows that
\begin{align*}
\mathrm{Ob}_{(\mathscr{D}_{\mathfrak{h}},\mathscr{D}_{\mathfrak{g}})\mathrm{I}}^{\hat{\mathfrak{g}},1}(x,y)= & \mathscr{D}_{\mathfrak{h}}(\psi_{1}(x,y))+\psi_{1}(\mathscr{D}_{\mathfrak{g}}(x),y)+\psi_{1}(x,\mathscr{D}_{\mathfrak{g}}(y))\\
= & \mathscr{D}_{\mathfrak{h}}(\psi_{2}(x,y))+\mathscr{D}_{\mathfrak{h}}(\kappa([x,y]_{\mathfrak{g}}))+\psi_{2}(\mathscr{D}_{\mathfrak{g}}(x),y)\\
 & +\kappa([\mathscr{D}_{\mathfrak{g}}(x),y]_{\mathfrak{g}})+\psi_{2}(x,\mathscr{D}_{\mathfrak{g}}(y))+\kappa([x,\mathscr{D}_{\mathfrak{g}}(y)]_{\mathfrak{g}})\\
= & \mathrm{Ob}_{(\mathscr{D}_{\mathfrak{h}},\mathscr{D}_{\mathfrak{g}})\mathrm{I}}^{\hat{\mathfrak{g}},2}(x,y)+(\mathscr{D}_{\mathfrak{h}}\text{\textopenbullet}\kappa+\kappa\text{\textopenbullet}\mathscr{D}_{\mathfrak{g}})([x,y]_{\mathfrak{g}})\\
= & \mathrm{Ob}_{(\mathscr{D}_{\mathfrak{h}},\mathscr{D}_{\mathfrak{g}})\mathrm{I}}^{\hat{\mathfrak{g}},2}(x,y)+d^{1}(\mathscr{D}_{\mathfrak{h}}\text{\textopenbullet}\kappa+\kappa\text{\textopenbullet}\mathscr{D}_{\mathfrak{g}})(x,y)
\end{align*}
and
\begin{align*}
\mathrm{Ob}_{(\mathscr{D}_{\mathfrak{h}},\mathscr{D}_{\mathfrak{g}})\mathrm{II}}^{\hat{\mathfrak{g}},1}(x) & =\mathscr{D}_{\mathfrak{h}}(\varsigma_{1}(x))+\psi_{1}(x,\mathscr{D}_{\mathfrak{g}}(x))\\
 & =\mathscr{D}_{\mathfrak{h}}(\varsigma_{2}(x)+\kappa(x^{[2]_{\mathfrak{g}}}))+\psi_{2}(x,\mathscr{D}_{\mathfrak{g}}(x))+\kappa([x,\mathscr{D}_{\mathfrak{g}}(x)]_{\mathfrak{g}})\\
 & =\mathscr{D}_{\mathfrak{h}}(\varsigma_{2}(x))+\mathscr{D}_{\mathfrak{h}}(\kappa(x^{[2]_{\mathfrak{g}}}))+\psi_{2}(x,\mathscr{D}_{\mathfrak{g}}(x))+\kappa(\mathscr{D}_{\mathfrak{g}}(x^{[2]_{\mathfrak{g}}}))\\
 & =\mathrm{Ob}_{(\mathscr{D}_{\mathfrak{h}},\mathscr{D}_{\mathfrak{g}})\mathrm{II}}^{\hat{\mathfrak{g}},2}(x)+(\mathscr{D}_{\mathfrak{h}}\text{\textopenbullet}\kappa+\kappa\text{\textopenbullet}\mathscr{D}_{\mathfrak{g}})(x^{[2]_{\mathfrak{g}}})\\
 & =\mathrm{Ob}_{(\mathscr{D}_{\mathfrak{h}},\mathscr{D}_{\mathfrak{g}})\mathrm{II}}^{\hat{\mathfrak{g}},2}(x)+\omega_{\mathscr{D}_{\mathfrak{h}}\text{\textopenbullet}\kappa+\kappa\text{\textopenbullet}\mathscr{D}_{\mathfrak{g}}}(x).
\end{align*}
Therefore, we have
\[
(\mathrm{Ob}_{(\mathscr{D}_{\mathfrak{h}},\mathscr{D}_{\mathfrak{g}})\mathrm{I}}^{\hat{\mathfrak{g}},1},\mathrm{Ob}_{(\mathscr{D}_{\mathfrak{h}},\mathscr{D}_{\mathfrak{g}})\mathrm{II}}^{\hat{\mathfrak{g}},1})-(\mathrm{Ob}_{(\mathscr{D}_{\mathfrak{h}},\mathscr{D}_{\mathfrak{g}})\mathrm{I}}^{\hat{\mathfrak{g}},2},\mathrm{Ob}_{(\mathscr{D}_{\mathfrak{h}},\mathscr{D}_{\mathfrak{g}})\mathrm{II}}^{\hat{\mathfrak{g}},2})=\partial^{1}(\mathscr{D}_{\mathfrak{h}}\text{\textopenbullet}\kappa+\kappa\text{\textopenbullet}\mathscr{D}_{\mathfrak{g}})\in B_{*2}^{2}(\mathfrak{g};\mathfrak{h}).
\]
Thus,
\[
[(\mathrm{Ob}_{(\mathscr{D}_{\mathfrak{h}},\mathscr{D}_{\mathfrak{g}})\mathrm{I}}^{\hat{\mathfrak{g}},1},\mathrm{Ob}_{(\mathscr{D}_{\mathfrak{h}},\mathscr{D}_{\mathfrak{g}})\mathrm{II}}^{\hat{\mathfrak{g}},1})]=[(\mathrm{Ob}_{(\mathscr{D}_{\mathfrak{h}},\mathscr{D}_{\mathfrak{g}})\mathrm{I}}^{\hat{\mathfrak{g}},2},\mathrm{Ob}_{(\mathscr{D}_{\mathfrak{h}},\mathscr{D}_{\mathfrak{g}})\mathrm{II}}^{\hat{\mathfrak{g}},2})]\in H_{*2}^{2}(\mathfrak{g};\mathfrak{h}).
\]

The proof is complete.
\end{proof}
\begin{definition} The cohomology class $[(\mathrm{Ob}_{(\mathscr{D}_{\mathfrak{h}},\mathscr{D}_{\mathfrak{g}})\mathrm{I}}^{\hat{\mathfrak{g}}},\mathrm{Ob}_{(\mathscr{D}_{\mathfrak{h}},\mathscr{D}_{\mathfrak{g}})\mathrm{II}}^{\hat{\mathfrak{g}}})]\in H_{*2}^{2}(\mathfrak{g};\mathfrak{h})$
is called the obstruction class of $(\mathscr{D}_{\mathfrak{h}},\mathscr{D}_{\mathfrak{g}})$
being extensible.

\end{definition}

\begin{theorem} Let $0\longrightarrow\mathfrak{h}\overset{i}{\longrightarrow}\hat{\mathfrak{g}}\overset{\theta}{\longrightarrow}\mathfrak{\mathfrak{g}}\longrightarrow0$
be a central extension of restricted Lie algebras. Then $(\mathscr{D}_{\mathfrak{h}},\mathscr{D}_{\mathfrak{g}})\in\mathrm{Der}^{2}(\mathfrak{h})\times\mathrm{Der}^{2}(\mathfrak{g})$
is extensible if and only if the obstruction class $[(\mathrm{Ob}_{(\mathscr{D}_{\mathfrak{h}},\mathscr{D}_{\mathfrak{g}})\mathrm{I}}^{\hat{\mathfrak{g}}},\mathrm{Ob}_{(\mathscr{D}_{\mathfrak{h}},\mathscr{D}_{\mathfrak{g}})\mathrm{II}}^{\hat{\mathfrak{g}}})]\in H_{*2}^{2}(\mathfrak{g};\mathfrak{h})$
is trivial.

\end{theorem}
\begin{proof}
Let $s:\mathfrak{g}\rightarrow\hat{\mathfrak{g}}$ be a section of
the central extension $0\longrightarrow\mathfrak{h}\overset{i}{\longrightarrow}\hat{\mathfrak{g}}\overset{\theta}{\longrightarrow}\mathfrak{\mathfrak{g}}\longrightarrow0$.

Suppose that $(\mathscr{D}_{\mathfrak{h}},\mathscr{D}_{\mathfrak{g}})$
is extensible. Then there exists a restricted derivation $\mathscr{D}_{\mathfrak{\hat{g}}}\in\mathrm{Der}^{2}(\hat{\mathfrak{g}})$
such that we have the exact sequence of ResLieDer pair morphisms in
Definition 6.1. Since $\mathscr{D}_{\mathfrak{g}}\text{\textopenbullet}\theta=\theta\text{\textopenbullet}\mathscr{D}_{\mathfrak{\hat{g}}}$,
we have $\mathscr{D}_{\mathfrak{\hat{g}}}(s(x))+s(\mathscr{D}_{\mathfrak{g}}(x))\in\mathfrak{h}$.
Define $\gamma:\mathfrak{g}\rightarrow\mathfrak{h}$ by $\gamma(x)=\mathscr{D}_{\mathfrak{\hat{g}}}(s(x))+s(\mathscr{D}_{\mathfrak{g}}(x))$.
Then for $s(x)+h\in\hat{\mathfrak{g}}$, we have
\begin{align*}
\mathscr{D}_{\hat{\mathfrak{g}}}(s(x)+h) & =\mathscr{D}_{\hat{\mathfrak{g}}}(s(x))+\mathscr{D}_{\mathfrak{h}}(h)=(\mathscr{D}_{\hat{\mathfrak{g}}}(s(x))+s(\mathscr{D}_{\mathfrak{g}}(x)))+s(\mathscr{D}_{\mathfrak{g}}(x))+\mathscr{D}_{\mathfrak{h}}(h)\\
 & =s(\mathscr{D}_{\mathfrak{g}}(x))+\gamma(x)+\mathscr{D}_{\mathfrak{h}}(h).
\end{align*}

Furthermore, for any $s(x)+h_{1},s(y)+h_{2}\in\hat{\mathfrak{g}}$,
we have
\begin{align*}
\mathscr{D}_{\hat{\mathfrak{g}}}([s(x)+h_{1},s(y)+h_{2}]_{\hat{\mathfrak{g}}}) & =\mathscr{D}_{\hat{\mathfrak{g}}}(([s(x),s(y)]_{\hat{\mathfrak{g}}}+s([x,y]_{\mathfrak{g}}))+s([x,y]_{\mathfrak{g}}))\\
 & =\mathscr{D}_{\hat{\mathfrak{g}}}(\psi(x,y)+s([x,y]_{\mathfrak{g}}))\\
 & =s(\mathscr{D}_{\mathfrak{g}}([x,y]_{\mathfrak{g}}))+\gamma([x,y]_{\mathfrak{g}})+\mathscr{D}_{\mathfrak{h}}(\psi(x,y))
\end{align*}
and
\begin{align*}
 & [\mathscr{D}_{\hat{\mathfrak{g}}}(s(x)+h_{1}),s(y)+h_{2}]_{\hat{\mathfrak{g}}}+[s(x)+h_{1},\mathscr{D}_{\hat{\mathfrak{g}}}(s(y)+h_{2})]_{\hat{\mathfrak{g}}}\\
= & [s(\mathscr{D}_{\mathfrak{g}}(x))+\gamma(x)+\mathscr{D}_{\mathfrak{h}}(h_{1}),s(y)+h_{2}]_{\hat{\mathfrak{g}}}+[s(x)+h_{1},s(\mathscr{D}_{\mathfrak{g}}(y))+\gamma(y)+\mathscr{D}_{\mathfrak{h}}(h_{2})]_{\hat{\mathfrak{g}}}\\
= & [s(\mathscr{D}_{\mathfrak{g}}(x)),s(y)]_{\hat{\mathfrak{g}}}+[s(x),s(\mathscr{D}_{\mathfrak{g}}(y))]_{\hat{\mathfrak{g}}}\\
= & ([s(\mathscr{D}_{\mathfrak{g}}(x)),s(y)]_{\hat{\mathfrak{g}}}+s([\mathscr{D}_{\mathfrak{g}}(x),y]_{\mathfrak{g}}))+s([\mathscr{D}_{\mathfrak{g}}(x),y]_{\mathfrak{g}})\\
 & +([s(x),s(\mathscr{D}_{\mathfrak{g}}(y))]_{\hat{\mathfrak{g}}}+s([x,\mathscr{D}_{\mathfrak{g}}(y)]_{\mathfrak{g}}))+s([x,\mathscr{D}_{\mathfrak{g}}(y)]_{\mathfrak{g}})\\
= & \psi(\mathscr{D}_{\mathfrak{g}}(x),y)+s([\mathscr{D}_{\mathfrak{g}}(x),y]_{\mathfrak{g}})+\psi(x,\mathscr{D}_{\mathfrak{g}}(y))+s([x,\mathscr{D}_{\mathfrak{g}}(y)]_{\mathfrak{g}}).
\end{align*}
It follows from
\[
\mathscr{D}_{\hat{\mathfrak{g}}}([s(x)+h_{1},s(y)+h_{2}]_{\hat{\mathfrak{g}}})=[\mathscr{D}_{\hat{\mathfrak{g}}}(s(x)+h_{1}),s(y)+h_{2}]_{\hat{\mathfrak{g}}}+[s(x)+h_{1},\mathscr{D}_{\hat{\mathfrak{g}}}(s(y)+h_{2})]_{\hat{\mathfrak{g}}}
\]
that
\[
\mathscr{D}_{\mathfrak{h}}(\psi(x,y))+\psi(\mathscr{D}_{\mathfrak{g}}(x),y)+\psi(x,\mathscr{D}_{\mathfrak{g}}(y))=\gamma([x,y]_{\mathfrak{g}}).
\]
It implies that $\mathrm{Ob}_{(\mathscr{D}_{\mathfrak{h}},\mathscr{D}_{\mathfrak{g}})\mathrm{I}}^{\hat{\mathfrak{g}}}=d^{1}\gamma$.
For any $s(x)+h\in\hat{\mathfrak{g}}$, we have
\begin{align*}
\mathscr{D}_{\hat{\mathfrak{g}}}((s(x)+h)^{[2]_{\hat{\mathfrak{g}}}}) & =\mathscr{D}_{\hat{\mathfrak{g}}}(s(x)^{[2]_{\hat{\mathfrak{g}}}})=\mathscr{D}_{\hat{\mathfrak{g}}}(s(x)^{[2]_{\hat{\mathfrak{g}}}}+s(x^{[2]_{\mathfrak{g}}})+s(x^{[2]_{\mathfrak{g}}}))\\
 & =\mathscr{D}_{\hat{\mathfrak{g}}}(\varsigma(x)+s(x^{[2]_{\mathfrak{g}}}))=s(\mathscr{D}_{\mathfrak{g}}(x^{[2]_{\mathfrak{g}}}))+\gamma(x^{[2]_{\mathfrak{g}}})+\mathscr{D}_{\mathfrak{h}}(\varsigma(x))
\end{align*}
and
\begin{align*}
[s(x)+h,\mathscr{D}_{\hat{\mathfrak{g}}}(s(x)+h)]_{\hat{\mathfrak{g}}} & =[s(x)+h,s(\mathscr{D}_{\mathfrak{g}}(x))+\gamma(x)+\mathscr{D}_{\mathfrak{h}}(h)]_{\hat{\mathfrak{g}}}\\
 & =[s(x),s(\mathscr{D}_{\mathfrak{g}}(x))]_{\hat{\mathfrak{g}}}+s([x,\mathscr{D}_{\mathfrak{g}}(x)]_{\mathfrak{g}})+s([x,\mathscr{D}_{\mathfrak{g}}(x)]_{\mathfrak{g}})\\
 & =\psi(x,\mathscr{D}_{\mathfrak{g}}(x))+s([x,\mathscr{D}_{\mathfrak{g}}(x)]_{\mathfrak{g}}).
\end{align*}
It follows from $\mathscr{D}_{\hat{\mathfrak{g}}}((s(x)+h)^{[2]_{\hat{\mathfrak{g}}}})=[s(x)+h,\mathscr{D}_{\hat{\mathfrak{g}}}(s(x)+h)]_{\hat{\mathfrak{g}}}$
that
\[
\mathscr{D}_{\mathfrak{h}}(\varsigma(x))+\psi(x,\mathscr{D}_{\mathfrak{g}}(x))=\gamma(x^{[2]_{\mathfrak{g}}})=\omega_{\gamma}.
\]
It implies that $\mathrm{Ob}_{(\mathscr{D}_{\mathfrak{h}},\mathscr{D}_{\mathfrak{g}})\mathrm{II}}^{\hat{\mathfrak{g}}}=\omega_{\gamma}$.
Therefore, we have
\[
(\mathrm{Ob}_{(\mathscr{D}_{\mathfrak{h}},\mathscr{D}_{\mathfrak{g}})\mathrm{I}}^{\hat{\mathfrak{g}}},\mathrm{Ob}_{(\mathscr{D}_{\mathfrak{h}},\mathscr{D}_{\mathfrak{g}})\mathrm{II}}^{\hat{\mathfrak{g}}})=\partial^{1}\gamma\in B_{*2}^{2}(\mathfrak{g};\mathfrak{h})
\]
Thus, the obstruction class $[(\mathrm{Ob}_{(\mathscr{D}_{\mathfrak{h}},\mathscr{D}_{\mathfrak{g}})\mathrm{I}}^{\hat{\mathfrak{g}}},\mathrm{Ob}_{(\mathscr{D}_{\mathfrak{h}},\mathscr{D}_{\mathfrak{g}})\mathrm{II}}^{\hat{\mathfrak{g}}})]\in H_{*2}^{2}(\mathfrak{g};\mathfrak{h})$
is trivial.

Conversely, if the obstruction class $[(\mathrm{Ob}_{(\mathscr{D}_{\mathfrak{h}},\mathscr{D}_{\mathfrak{g}})\mathrm{I}}^{\hat{\mathfrak{g}}},\mathrm{Ob}_{(\mathscr{D}_{\mathfrak{h}},\mathscr{D}_{\mathfrak{g}})\mathrm{II}}^{\hat{\mathfrak{g}}})]\in H_{*2}^{2}(\mathfrak{g};\mathfrak{h})$
is trivial, then there exists a map $\gamma:\mathfrak{g}\rightarrow\mathfrak{h}$
such that $(\mathrm{Ob}_{(\mathscr{D}_{\mathfrak{h}},\mathscr{D}_{\mathfrak{g}})\mathrm{I}}^{\hat{\mathfrak{g}}},\mathrm{Ob}_{(\mathscr{D}_{\mathfrak{h}},\mathscr{D}_{\mathfrak{g}})\mathrm{II}}^{\hat{\mathfrak{g}}})=\partial^{1}\gamma$.
For any $s(x)+h\in\hat{\mathfrak{g}}$, we define $\mathscr{D}_{\hat{\mathfrak{g}}}$
by
\[
\mathscr{D}_{\hat{\mathfrak{g}}}(s(x)+h)=s(\mathscr{D}_{\mathfrak{g}}(x))+\gamma(x)+\mathscr{D}_{\mathfrak{h}}(h).
\]
It can be proved that $\mathscr{D}_{\hat{\mathfrak{g}}}$ is a restricted
derivation of $\hat{\mathfrak{g}}$ such that the diagram in Definition
6.1 commutes. Therefore, $(\mathscr{D}_{\mathfrak{h}},\mathscr{D}_{\mathfrak{g}})$
is extensible.

The proof is complete.
\end{proof}
Clearly, it can be obtained that

\begin{corollary} Let $0\longrightarrow\mathfrak{h}\overset{i}{\longrightarrow}\hat{\mathfrak{g}}\overset{\theta}{\longrightarrow}\mathfrak{\mathfrak{g}}\longrightarrow0$
be a central extension of restricted Lie algebras. If $H_{*2}^{2}(\mathfrak{g};\mathfrak{h})=0$,
then any pair $(\mathscr{D}_{\mathfrak{h}},\mathscr{D}_{\mathfrak{g}})\in\mathrm{Der}^{2}(\mathfrak{h})\times\mathrm{Der}^{2}(\mathfrak{g})$
is extensible.

\end{corollary}

At the end of the section, the conditions on a pair $(\mathscr{D}_{\mathfrak{h}},\mathscr{D}_{\mathfrak{g}})\in\mathrm{Der}^{2}(\mathfrak{h})\times\mathrm{Der}^{2}(\mathfrak{g})$
being extensible in any central extension of $(\mathfrak{g},[2]_{\mathfrak{g}})$
by $(\mathfrak{h},[2]_{\mathfrak{h}})$ are considered. According
to Proposition 6.2, we may define a linear map $\Phi:\mathrm{Der}^{2}(\mathfrak{h})\times\mathrm{Der}^{2}(\mathfrak{g})\rightarrow\mathfrak{gl}(H_{*2}^{2}(\mathfrak{g};\mathfrak{h}))$
by
\[
\Phi(\mathscr{D}_{\mathfrak{h}},\mathscr{D}_{\mathfrak{g}})([(\psi,\varsigma)])=[(\mathscr{D}_{\mathfrak{h}}\text{\textopenbullet}\psi+\psi(\mathscr{D}_{\mathfrak{g}}\land\mathrm{Id})+\psi(\mathrm{Id}\land\mathscr{D}_{\mathfrak{g}}),\mathscr{D}_{\mathfrak{h}}\text{\textopenbullet}\varsigma+\psi(\mathrm{Id}\land\mathscr{D}_{\mathfrak{g}}))].
\]

\begin{theorem} Let $(\mathfrak{h},[2]_{\mathfrak{h}})$ be a strongly
abelian restricted Lie algebra and $(\mathfrak{g},[2]_{\mathfrak{g}})$
a restricted Lie algebra. A pair $(\mathscr{D}_{\mathfrak{h}},\mathscr{D}_{\mathfrak{g}})\in\mathrm{Der}^{2}(\mathfrak{h})\times\mathrm{Der}^{2}(\mathfrak{g})$
is extensible in any central extension of $\mathfrak{g}$ by $\mathfrak{h}$
if and only if $\Phi(\mathscr{D}_{\mathfrak{h}},\mathscr{D}_{\mathfrak{g}})=0$.

\end{theorem}
\begin{proof}
Suppose that $\Phi(\mathscr{D}_{\mathfrak{h}},\mathscr{D}_{\mathfrak{g}})=0$.
For any central extension $0\longrightarrow\mathfrak{h}\overset{i}{\longrightarrow}\hat{\mathfrak{g}}\overset{\theta}{\longrightarrow}\mathfrak{\mathfrak{g}}\longrightarrow0$,
we choose a section $s:\mathfrak{g}\rightarrow\hat{\mathfrak{g}}$.
Then the pair $(\psi,\varsigma)$, where $\psi:\mathfrak{g}\land\mathfrak{g}\rightarrow\mathfrak{h}$
and $\varsigma:\mathfrak{g}\rightarrow\mathfrak{h}$ are respectively
defined by
\begin{align*}
\psi(x,y) & =[s(x),s(y)]_{\hat{\mathfrak{g}}}+s([x,y]_{\mathfrak{g}}),\\
\varsigma(x) & =s(x)^{[2]_{\hat{\mathfrak{g}}}}+s(x^{[2]_{\mathfrak{g}}}),
\end{align*}
is a 2-cocycle. Furthermore, we have
\begin{align*}
[(\mathrm{Ob}_{(\mathscr{D}_{\mathfrak{h}},\mathscr{D}_{\mathfrak{g}})\mathrm{I}}^{\hat{\mathfrak{g}}},\mathrm{Ob}_{(\mathscr{D}_{\mathfrak{h}},\mathscr{D}_{\mathfrak{g}})\mathrm{II}}^{\hat{\mathfrak{g}}})] & =[(\mathscr{D}_{\mathfrak{h}}\text{\textopenbullet}\psi+\psi(\mathscr{D}_{\mathfrak{g}}\land\mathrm{Id})+\psi(\mathrm{Id}\land\mathscr{D}_{\mathfrak{g}}),\mathscr{D}_{\mathfrak{h}}\text{\textopenbullet}\varsigma+\psi(\mathrm{Id}\land\mathscr{D}_{\mathfrak{g}}))]\\
 & =\Phi(\mathscr{D}_{\mathfrak{h}},\mathscr{D}_{\mathfrak{g}})([(\psi,\varsigma)])=0.
\end{align*}
According to Theorem 6.4, $(\mathscr{D}_{\mathfrak{h}},\mathscr{D}_{\mathfrak{g}})$
is extensible in the above central extension.

Conversely, for any $[(\psi,\varsigma)]\in H_{*2}^{2}(\mathfrak{g};\mathfrak{h})$,
there exists a central extension $0\longrightarrow\mathfrak{h}\overset{i}{\longrightarrow}\mathfrak{g}\oplus\mathfrak{h}\overset{\theta}{\longrightarrow}\mathfrak{\mathfrak{g}}\longrightarrow0$,
where the Lie bracket and 2-mapping on $\mathfrak{g}\oplus\mathfrak{h}$
are respectively defined by
\begin{align*}
[x+h_{1},y+h_{2}] & =[x,y]_{\mathfrak{g}}+\psi(x,y),\\
(x+h)^{[2]} & =x^{[2]_{\mathfrak{g}}}+\varsigma(x).
\end{align*}
Since $(\mathscr{D}_{\mathfrak{h}},\mathscr{D}_{\mathfrak{g}})$ is
extensible in any central extension of $\mathfrak{g}$ by $\mathfrak{h}$.
According to Theorem 6.4, we have
\begin{align*}
\Phi(\mathscr{D}_{\mathfrak{h}},\mathscr{D}_{\mathfrak{g}})([(\psi,\varsigma)]) & =[(\mathscr{D}_{\mathfrak{h}}\text{\textopenbullet}\psi+\psi(\mathscr{D}_{\mathfrak{g}}\land\mathrm{Id})+\psi(\mathrm{Id}\land\mathscr{D}_{\mathfrak{g}}),\mathscr{D}_{\mathfrak{h}}\text{\textopenbullet}\varsigma+\psi(\mathrm{Id}\land\mathscr{D}_{\mathfrak{g}}))]\\
 & =[(\mathrm{Ob}_{(\mathscr{D}_{\mathfrak{h}},\mathscr{D}_{\mathfrak{g}})\mathrm{I}}^{\hat{\mathfrak{g}}},\mathrm{Ob}_{(\mathscr{D}_{\mathfrak{h}},\mathscr{D}_{\mathfrak{g}})\mathrm{II}}^{\hat{\mathfrak{g}}})]=0.
\end{align*}
Thus, $\Phi(\mathscr{D}_{\mathfrak{h}},\mathscr{D}_{\mathfrak{g}})=0$.

The proof is complete.
\end{proof}

\end{document}